\documentclass[11pt,a4paper,reqno]{amsart}
\usepackage[margin=3cm]{geometry}
\usepackage{color}               
\usepackage{faktor}
\usepackage{tikz-cd}
\usepackage{enumitem}

\usepackage{hyperref}

\usepackage{amsmath,amssymb,anyfontsize,stmaryrd,mathtools, thmtools,tikz}
\usepackage{mathrsfs}
\usepackage[oldstyle]{libertine}%
\usepackage[T1]{fontenc}
\usepackage[utf8]{inputenc}
\usepackage{csquotes}
\usepackage{tikz-cd}
\usepackage[noabbrev]{cleveref}
\usepackage{autonum}
\usepackage{url}
\crefname{lemma}{lemma}{lemmata}
\Crefname{lemma}{Lemma}{Lemmata}
\crefname{subsection}{subsection}{subsections}
\Crefname{subsection}{Subsection}{Subsections}

\newtheorem{theorem}{Theorem}[section]
\newtheorem{lemma}[theorem]{Lemma}
\newtheorem{proposition}[theorem]{Proposition}
\newtheorem{corollary}[theorem]{Corollary}

\theoremstyle{definition}
\newtheorem{definition}[theorem]{Definition}
\newtheorem{remark}[theorem]{Remark}

\newtheorem{example}[theorem]{Example}

\newcommand{\mb}[1]{\mathbb{#1}}

\newcommand{\C}{\mb{C}} 
\newcommand{\Z}{\mb{Z}}

\newcommand{\Q}{\mb{Q}}

  \newcommand{\rk}{\mathrm{rank}}

\newcommand{\Ch}{{\mathfrak{ C}}}
\newcommand{\Oh}{{\mathcal O}}

\newcommand{\Syz}{\mathrm{Syz}}

\newcommand{\cE}{\mathcal{E}}

\newcommand{\cX}{\mathcal{X}}

\newcommand\restr[2]{{
  \left.\kern-\nulldelimiterspace 
  #1 
  \vphantom{\big|} 
  \right|_{#2} 
  }}

\newcommand\restrst[2]{{
  \left.\kern-\nulldelimiterspace 
  #1 
  \vphantom{\big|} 
  \right|^*_{#2} 
  }}

\begingroup\expandafter\expandafter\expandafter\endgroup
\expandafter\ifx\csname pdfsuppresswarningpagegroup\endcsname\relax
\else
\fi
{}

\title{Strongly semistable reduction of syzygy bundles on plane curves}
\author[M.~A.~Hahn]{Marvin Anas Hahn}
\address{M.~A.~Hahn: Institut für Mathematik, Goehte-Universität Frankfurt, Robert-Mayer-Str. 6-8, 60325 Frankfurt am Main}
\email{hahn@math.uni-frankfurt.de}
\author[A.~Werner]{Annette Werner}
\address{A.~Werner: Institut für Mathematik, Goehte-Universität Frankfurt, Robert-Mayer-Str. 6-8, 60325 Frankfurt am Main}
\email{werner@math.uni-frankfurt.de}


\begin{document}
\begin{abstract}
We investigate degenerations of syzygy bundles on plane curves over $p$-adic fields. We use Mustafin varieties which are degenerations of projective spaces 
to find a large family of models of  plane curves over the ring of integers such that the special fiber consists of multiple projective lines meeting in one point. On such models we investigate vector bundles whose generic fiber is a syzygy bundle and which become trivial when restricted to each projective line in the special fiber. Hence these syzygy bundles have strongly semistable reduction. This investigation is motivated by the fundamental open problem in $p$-adic Simpson theory to determine the category of Higgs bundles corresponding to continuous representations of the \'etale fundamental group of a curve. Faltings' $p$-adic Simpson correspondence  and work of Deninger and the second author shows that bundles with Higgs field zero and potentially strongly semistable reduction fall into this category. 
Hence the results in the present paper determine a class of syzygy bundles on plane curves giving rise to a $p$-adic local system.  We apply our methods to a concrete example on the Fermat curve suggested by Brenner and prove that this bundle has potentially strongly semistable reduction. 
\end{abstract}
\maketitle

\section{Introduction}
The classical Simpson correspondence in dimension one establishes a correspondence between semistable degree zero Higgs bundles on a Riemann surface $X$ and representations of its topological fundamental group  \cite{simpson1990nonabelian,simpson1992higgs}. In recent years considerable progress has been made towards a similar result in the $p-$adic case \cite{dewe1,dewe4,fa,abbes2016p, liuzhu}. In \cite{fa}, Faltings proved an equivalence of categories between Higgs bundles on a $p-$adic curve and so-called \textit{generalized representations} of its étale fundamental group. A detailed and systematic treatment of the local theory is provided by \cite{abbes2016p}. Recently Liu and Zhu \cite{liuzhu} have  established a Riemann-Hilbert functor on rigid analytic varieties which yields part of a $p$-adic Simpson correspondence, namely a tensor functor from the category of \'etale $\Q_p$-local systems to the category of nilpotent Higgs bundles.  

\par 
An approach for Higgs bundles with trivial Higgs fields, which may be seen as $p$-adic analog of the classical Narasimhan-Seshadri correspondence, was introduced by Deninger and the second author. It is compatible with Faltings' functor by \cite{xu}. In \cite{dewe1, dewe4} it is shown that a semistable vector bundle on a proper, smooth $p$-adic curve $X$ which has strongly semistable reduction of degree zero after pullback to a finite covering of the curve  admits $p$-adic parallel transport and hence gives rise to a continuous representation of the \'etale fundamental group. See also \cite{ dewe3} for a more detailed analysis of the Tannaka groups involved. In  \cite{dewe5} this result is generalized to bundles with numerically flat reduction on $p$-adic varieties of any dimension. 

In order to obtain a complete picture of a  $p$-adic Simpson theory on curves, we have to determine the category of vector bundles  on a $p$-adic curve which is equivalent to the category of  continuous representations of the \'etale fundamental group. One is of course tempted to speculate that  it is the category of semistable Higgs bundles of degree zero. By the results of Deninger and the second author a positive answer for semistable degree zero bundles with trivial Higgs fields can be given if we  prove a potentially strongly semistable reduction theorem for all such bundles. This involves finding suitable models of both the curve and the vector bundle and is therefore, as might be expected, a difficult task.\par 

Since we know that all line bundles of degree zero lie in the category of degree zero bundles with potentially strongly semistable reduction, and since this category is closed under extensions, the first interesting case is provided by  stable rank two bundles of degree zero. Here 
Holger Brenner \cite{br2} has  given a concrete example of a stable rank two syzygy bundle on a Fermat curve which does not have strongly semistable reduction in the obvious way, i.e. by simply degenerating the given equations for bundle and curve which are defined over the ring of integers.\par

This motivates the study of szygy bundles on plane curves over discretely valued fields as a first step. The present paper proposes the use of Mustafin varieties to define suitable models of plane curve. Mustafin varieties are flat degenerations of projective spaces induced by a finite choice of invertible matrices, which have been introduced and studied in \cite{mustafin1978nonarchimedean,cartwright2011mustafin, hahn2017mustafin} 
Our first main result \cref{thm:spec} 
shows that for sufficiently general choice of these invertible matrices, and if the residue field is big enough, the model of a plane curve obtained taking the closure in the Mustafin variety has
star-like reduction, i.e. its special fiber consists of multiple projective lines meeting in one point. It turns out (see \cref{lem:step1}) that it suffices to arrange the situation in such a way that all irreducible components of the model of the curve are contained in primary components of the Mustafin variety. 

Then we attack the problem of extending syzygy bundles on projective planes to sheaves on Mustafin varieties. We define such an extension which mixes information from the different primary components and investigate the locus where it is locally free. If the equations of the syzygy bundle satisfy a certain condition the restriction of this sheaf will be locally free on the model of the curve with strongly semistable special fiber, as we show in our second main result \cref{thm:plane}.

 Finally, we apply this result to Brenner's example. i.e. to the  bundle $Syz(x^2,y^2,z^2)$ on a Fermat curve $X$. We show in \cref{thm:bren} that there exists a (ramified) covering of $X$, such that the pullback of $Syz(x^2,y^2,z^2)$ to this cover has a model with strongly semistable reduction.  Since we have to choose the covering and the model at the same time, we have to adapt the arguments of \cref{thm:plane} to this situation. Therefore, Brenner's example  $Syz(x^2,y^2,z^2)$ from  \cite{br2} admits indeed \'etale parallel transport since there exists  a finite cover of the Fermat curve where its pullback has strongly semistable reduction of degree zero.\par
 Our method shows that the degree of freedom provided in the choice of models for the curve and the bundle is quite big. In particular, our results strengthen the hope that every semistable vector bundle of degree zero on a $p$-adic curve with Higgs field zero participates in the $p$-adic Simpson correspondence.\par  

\subsection{Structure of this paper}
In \cref{sec:pre}, we recall some background regarding Mustafin varieties, syzygy bundles and semistability of vector bundles. 
In \cref{sec:mus}, we construct models of plane curves with star-like reduction and prove \cref{thm:spec}. In \cref{sec:plane}, we state and and prove our second main \cref{thm:plane} that given a model with star-like reduction, a certain class of syzygy bundles can be extended to vector bundles with strongly semistable reduction of degree zero. Finally, we combine our methods in \cref{sec:brenner} to show that $Syz(x^2, x^2, z^2)$ does admit $p$-adic \'etale parallel transport by providing a finite cover of the Fermat curve where it has strongly semistable reduction of degree zero.

\subsection*{Acknowledgements} The authors gratefully acknowledge support of the LOEWE research unit Uniformized Structures in Arithmetic and Geometry. Many computations for this projects were aided by \textsc{Singular} \cite{DGPS}.

\section{Preliminaries}
\label{sec:pre}
In this section, we recall some background needed for this work. For the rest of this paper, we fix a discretely valued field $K$ with ring of integers denoted by $R$ or $R_K$, maximal ideal $\mathfrak{m}$ and perfect residue field $k$. Let $t$ be a fixed uniformizer of $R$. 

\subsection{Mustafin varieties}
\label{sec:must}
We introduce the basic notions surrounding Mustafin varieties, for more details, we refer to \cite{cartwright2011mustafin,hahn2017mustafin}. Let $V$ be a vector space of dimension $d$ over $K$. We define $\mathbb{P}(V)=\mathrm{Proj}\mathrm{Sym}(V^*)$ as the projective space parameterising lines in $V$. We call free $R-$modules $L\subset V$ of rank $d$ \textit{lattices} and define $\mathbb{P}(L)=\mathrm{Proj}\mathrm{Sym}(L^*)$, where $L^*=\mathrm{Hom}_R(L,R)$. Note that we will mostly consider lattices up to homothety, i.e. $L\backsim L'$ if $L=c\cdot L'$ for some $c\in K^{\times}$.

\begin{definition}
\label{def:musta}
Let $\Gamma=\{L_1,\dots,L_n\}$ be a set of rank $d$ lattice classes in $V$. Then $\mathbb{P}(L_1),\dots,\allowbreak\mathbb{P}(L_n)$ are projective spaces over $R$ whose generic fibers are canonically isomorphic to $\mathbb{P}(V)\simeq\mathbb{P}^{d-1}_{K}$. The open immersions
\begin{equation}
\mathbb{P}(V)\hookrightarrow\mathbb{P}(L_i)
\end{equation}
give rise to a map
\begin{equation}
\mathbb{P}(V)\longrightarrow\mathbb{P}(L_1)\times_R\dots\times_R\mathbb{P}(L_n).
\end{equation}
We denote the closure of the image endowed with the reduced scheme structure by $\mathcal{M}(\Gamma)$. We call $\mathcal{M}(\Gamma)$ the \textit{associated Mustafin variety}. Its special fiber $\mathcal{M}(\Gamma)_k$ is a reduced scheme over $k$ by \cite[Theorem 2.3]{cartwright2011mustafin}.\end{definition}

Let $L=Re_1+\dots+Re_d$ be a reference lattice.  By the following procedure, we choose coordinates on $\mathbb{P}(L_1)\times_R\dots\times_R\mathbb{P}(L_n)$: Let $g_i\in\mathrm{PGL}(V)$, such that $g_iL=L_i$. We consider the commutative diagram
\begin{equation}
\begin{tikzcd}
\mathbb{P}(V)\arrow{rr}{(g_1^{-1},\dots,g_n^{-1})\circ\Delta} \arrow{d} &  & \mathbb{P}(V)^n\arrow{d}\\
\prod_{R}\mathbb{P}(L_i) \arrow{rr}{(g_1^{-1},\dots,g_n^{-1})} &  & \mathbb{P}(L)^n.
\end{tikzcd}
\end{equation}
Let $x_1,\dots,x_d$ be the coordinates on $\mathbb{P}(L)$ and consider the projections
\begin{equation}
P_j:\mathbb{P}(L)^n\to\mathbb{P}(L)
\end{equation}
to the $j-$th factor. Then, we denote $x_{ij}=P_j^*x_i$ and observe that the Mustafin variety $\mathcal{M}(\Gamma)$ is isomorphic to the subscheme of $\mathbb{P}(L)^n$ cut out by
\begin{equation}
I_2\begin{pmatrix}
g_1\begin{pmatrix}
x_{11}\\
\vdots\\
x_{d1}
\end{pmatrix} &
\cdots
& g_n\begin{pmatrix}
x_{1n}\\
\vdots\\
x_{dn}
\end{pmatrix}
\end{pmatrix}\cap R[(x_{ij})].
\end{equation}
By
\begin{equation}
p_j=\restr{P_j}{\mathcal{M}(\Gamma)}:\mathcal{M}(\Gamma)\hookrightarrow\mathbb{P}(L)^n\to\mathbb{P}(L)
\end{equation}
we denote the projection to the $j-$th component. We write $x_{ij}$ also for the induced rational function on $\mathcal{M}(\Gamma)$. By \cite[Corollary 2.5]{cartwright2011mustafin}, for each $i$ there exists a unique irreducible component $X$ of $\mathcal{M}(\Gamma)_k$ which maps birationally onto $\mathbb{P}(L)_k$ via the map on the special fiber induced by $p_i$ We call $X$ the $i-$\textit{th primary component} of $\mathcal{M}(\Gamma)_k$.

\subsection{Syzygy bundles}
We consider syzygy sheaves on the projective space which are the kernel of a morphism to the structure sheaf. To be precise, let $f_1, \ldots, f_{n+1}$ be homogeneous polynomials in $K[x_1,\dots,x_N]$ with degrees $d_1, \ldots, d_{n+1}$. Then the corresponding syzygy sheaf $\Syz(f_1, \ldots, f_{n+1})$ on $\mathbb{P}_K^{N-1}$ is defined as the kernel

\[ 0 \longrightarrow \Syz(f_1, \ldots, f_{n+1}) \longrightarrow \bigoplus_{i=1}^{n+1} \Oh(-d_i) \xrightarrow{(f_1, \ldots, f_{n+1})}\mathcal{O}.\]
The sheaf $\Syz(f_1, \ldots, f_{n+1})$ is locally free on $\bigcup D_+(f_i)$. 

In this work, we will be concerned with vector bundles of degree zero on curves. Therefore, we consider the twisted sheaves $\Syz(f_1, \ldots, f_{n+1})(\rho)$ when $\sum d_i=n\rho$.

\begin{remark}
We note that usually a coherent sheaf $\mathcal{F}$ on $X$ is called a $k-$\textit{th syzygy sheaf} if for each $x\in X$, there exist an open neighbourhood $U$ of $x$, locally free sheaves $\mathcal{G}_1,\dots,\mathcal{G}_k$ on $U$ and an exact sequence
\begin{equation}
0\to\restr{\mathcal{F}}{U}\to\mathcal{G}_1\to\dots\to\mathcal{G}_k.
\end{equation}
\end{remark}

Thus the sheaf $\Syz(f_1, \ldots, f_{n+1})$ is a second syzygy sheaf.

\subsection{Semistability of vector bundles}
Recall that a vector bundle $E$  on  a smooth, projective and connected curve $C$ over a field $\kappa$ is semistable (respectively stable), if for all proper non-zero subbundles $F$ of $E$ the inequality $\deg (F)/\rk(F)  \le \deg (E)/\rk(E)$ (respectively $\deg (F)/\rk(F) < \deg (E)/\rk(E) $) holds.

If $\kappa$ has positive characteristic, semistability has weaker properties than in characteristic zero, since this property may be lost under pullback by inseparable morphism. This explains the following notion of strong semistability.

Assume that $\mbox{char} (\kappa) = p>0$, and let
$F : C \to C$ be the absolute Frobenius morphism, defined by the $p$-power map
on the structure sheaf. 
Then a vector bundle $E$ on $C$ is called strongly semistable, 
if $F^{n*} E$ is semistable on $C$ for all $n
\ge 1$.

\begin{definition}
\label{def:strong}
Let $E$ be a vector bundle on a one-dimensional proper scheme $C$ over a field $\kappa$ of characteristic $p$. Then $E$ is called strongly semistable of degree zero, if the pullback of $E$ to all normalized irreducible components of $C$ is strongly semistable of degree zero.
\end{definition}

\subsection{Parallel transport for $p$-adic vector bundles}
\label{sec:parallel}
Consider a smooth, projective and connected curve $C$ over $\overline{\mathbb{Q}}_p$, and denote by $C_{\C_p}$ the base change to the field $\C_p$ (which is the completion of the algebraic closure $\overline{\mathbb{Q}}_p$). By $\mathfrak{o}$ we denote the ring of integers of $\C_p$. Its residue field is isomorphic to $\overline{\mathbb{F}}_p$. We call every finitely presented, flat and proper $\overline{\Z}_p$-scheme $\mathfrak{C}$ with generic fiber $C$ a model of $C$.

\begin{definition}
\label{def:strongred}
A vector bundle $E$ on $C_{\C_p}$ has strongly semistable reduction of degree zero, if there exists a model $\Ch$ of $C$ and a vector bundle $\mathcal{E}$ on $\Ch_{\mathfrak{o}} = \Ch \otimes_{\overline{\mathbb{Z}}_p} \mathfrak{o}$ such that $\mathcal{E}$ has generic fiber $E_{\mathbb{C}_p}$ and such that the special fiber  $\mathcal{E}_{\overline{\mathbb{F}}_p}$ of $\mathcal{E}$ is strongly semistable of degree zero on the one-dimensional proper scheme $\Ch \otimes_{\overline{\mathbb{Z}}_p} \mathbb{F}_p$ in the sense of  \cref{def:strong}.
\end{definition}

In \cite{dewe1} and \cite{dewe3}, a theory of parallel transport along \'etale paths is defined for those vector bundles $E$ of degree zero on $C_{\C_p}$ for which there exists a finite, \'etale covering $\alpha: C' \rightarrow C$ such that the  bundle $\alpha_{\C_p}^\ast E$ on $C'_{\C_p}$ has strongly semistable reduction of degree zero. We note that if $E$ has strongly semistable reduction of degree zero, then $E$ is semistable of degree zero \cite[Theorem 13]{dewe1}.

\begin{definition}
\label{def:potstrongred}
A vector bundles $E$ of degree zero on $C_{\C_p}$ has \textit{potentially strongly semistable reduction} if there exists a finite (not necessarily \'etale) covering $\alpha: C' \rightarrow C$ such that the bundle $\alpha^\ast_{\C_p} E$ on $C'_{\C_p}$ has strongly semistable reduction.
\end{definition}

It is an important open question if all semistable bundles of degree zero on $C_{\C_p}$ have potentially strongly semistable reduction in this sense. In fact,  \cite[Theorem 10]{dewe4} implies that all bundles with potentially strongly semistable reduction admit $p$-adic parallel transport. Hence, using \cite{xu} and \cite{fa}, a positive answer to this question would imply that all semistable bundles of degree zero on $C_{\C_p}$ with trivial Higgs field correpond to $p$-adic representations of the \'etale fundamental group under the $p$-adic Simpson correspondence, which would represent a big step in the directon of a $p$-adic result which is analogous to the classical Simpson correspondence.

\section{Mustafin degenerations of plane curves}
\label{sec:mus}
In this section, we construct models of plane curves using Mustafin varieties. We begin by choosing a specific Mustafin variety. As we are only concerned with plane curves, we focus on the following situation:\par 
We denote by $\Delta$ the open subvariety of the affine space $\mathbb{A}_{\mathbb{Z}}^{9(n+1)}$ given by all points  $\underline{x}=\left(x_{ij}^{(l)}\right)_{i,j=1,2,3;l=1,\dots,n+1}$ such that $\det \left(x_{ij}^{(l)}\right)_{i,j} $ is invertible for all $l$. 

Let $\underline{a}=\left(a_{ij}^{(l)}\right)_{i,j=1,2,3;l=1,\dots,n+1}\in \Delta(K)$ be some tuple  of matrix coefficients, and put $V=K^3$. We choose the standard basis $e_1,e_2,e_3$ of $K^3$ and set $L=R e_1+R e_2+R e_3$ as the reference lattice. 
We define
\begin{equation}
M_l=\begin{pmatrix}
a_{11}^{(l)} & a_{12}^{(l)} & a_{13}^{(l)}\\
a_{21}^{(l)} & a_{22}^{(l)} & a_{23}^{(l)}\\
a_{31}^{(l)} & a_{32}^{(l)} & a_{33}^{(l)}
\end{pmatrix}
\end{equation}
and
\begin{equation}
g_l=M_l\begin{pmatrix}
1 & 0 & 0\\
0 & t & 0\\
0 & 0 & t^2
\end{pmatrix}
\end{equation}
for $l=1,\dots,n+1$.

This gives us the lattices $L_i=g_iL$ and the set $\Gamma=\{L_1,\dots,L_{n+1}\}$. We denote the corresponding Mustafin variety in $ \mathbb{P}(L)^{n+1}$ by $\mathcal{M}_{\underline{a}}(\Gamma)$. Let $C \subset\mathbb{P}^2_{K}$ be an irreducible plane curve. We embed it into $\mathcal{M}_{\underline{a}}(\Gamma)\subset \mathbb{P}(L)^{n+1}$ via
\begin{equation}
\label{equ:mustmod}
 C\subset\mathbb{P}^2_{K}\xrightarrow{(g_1^{-1},\dots,g_{n+1}^{-1})}\mathbb{P}(L)^{n+1}\end{equation}
and consider the closure of $C$ in $\mathcal{M}_{\underline{a}}(\Gamma)$ endowed with the reduced scheme structure. By the same considerations as in the proof of \cite[Theorem 2.3]{cartwright2011mustafin}, this yields a flat proper $R_{}-$scheme $\mathfrak{C}^{\underline{a}}$ with generic fiber $C$, which we also call a Mustafin model of $C$. We further denote its special fiber by $\mathfrak{C}^{\underline{a}}_{k}$ and the irreducible components of $\mathfrak{C}^{\underline{a}}_{k}$ by $C_1,\dots,C_m$. Note, that all $C_i$ are of dimension $1$ by \cite[32.19.2]{SP}.\par We write $D_i = C_i^{\mathrm{red}}$ for the corresponding reduced irreducible components. Our next aim to describe the irreducible components of the special fiber $\mathfrak{C}^{\underline{a}}_{k}$.
To begin with, we compute the components of the special fiber of the Mustafin variety.

\begin{definition}
\label{rem:general}
We say that a  condition holds for general elements $(a_{ij}^{(l)})_{i,j=1,2,3;l=1,\dots,n+1}\in R ^{9 (n+1)}$, if it holds for all elements in the preimage of  a non-empty Zariski open subset in ${k}^{9(n+1)}$ under the reduction map. In particular, a condition holding for general elements is generically true in $R^{9(n+1)}$.\par
Moreover, let $U\subset \mathbb{A}_K^{9(n+1)}$ be a non-empty Zariski open subset, then, possibly after replacing $K$ by a finite field extension,  $U(K) \cap R^{9(n+1)}$ contains the preimage of a non-empty Zariski open subset in $k^{9(n+1)}$, i.e. it contains a general subset. 
\end{definition}
 In the following, we compute the special fiber of the Mustafin variety considered above.

\begin{lemma}
\label{lem:prime}
Assume that the residue field $k$ is algebraically closed and $n\ge2$. 
For general $\underline{a} \in \Delta(R)$ the $l-$th primary component of $\mathcal{M}_{\underline{a}}(\Gamma)$ is cut out by the ideal
\begin{equation}
\label{equ:prim}
\langle \left(x_{1j},x_{2j}\right)_{j=1,\dots,n+1;j\neq l}\rangle
\end{equation}

Moreover, the secondary components are cut out by the ideal

\begin{equation}
\label{equ:sec}
\langle x_{1l},x_{1i},\left(x_{1j},x_{2j}\right)_{j=1,\dots,n+1;j\neq i,l}\rangle
\end{equation}
for $i,l=1,\dots,n+1,i\neq l$.
\end{lemma}

\begin{proof}
Generalizing the observation \cite[Example 2.2]{cartwright2011mustafin}, we find that for general $\underline{a}$, the special fiber of $\mathcal{M}_{\underline{a}}(\{L_l,L_m,L_o\})$ is cut out by
\begin{align}
&\langle x_{1m},x_{2m},x_{1o},x_{2o} \rangle\cap\langle x_{1l},x_{2l},x_{1o},x_{2o}\rangle\cap\langle x_{1l},x_{2l},x_{1m}x_{2m}\rangle\\
&\langle x_{1l},x_{1m},x_{1o},x_{2o}\rangle\cap\langle x_{1l},x_{2l},x_{1m},x_{1o}\rangle\cap\langle x_{1l},x_{1o},x_{1m},x_{2m}\rangle.
\end{align}
We now consider the projection
\begin{align}
p_{mno}:\mathcal{M}_{\underline{a}}({\Gamma})\to\mathcal{M}_{\underline{a}}(\{{L}_m,{L}_n,{L}_o\})
\end{align}
on the $m-$th, $n-$th and $o-$th factor and note that the following diagram commutes
\begin{equation}
\begin{tikzcd}
\mathbb{P}^2_K \arrow{rr}{(g_1^{-1},\dots,g_{n+1}^{-1})} \arrow{d}{\mathrm{id}}&  & \mathcal{M}_{\underline{a}}(\Gamma) \arrow{dr}{p_m}\arrow{d}{p_{mno}}\\
\mathbb{P}^2_K \arrow{rr}{(g_m^{-1},g_n^{-1},g_{o}^{-1})} &  & \mathcal{M}_{\underline{a}}(L_m,L_n,L_o\})\arrow{r}{p_m} & \mathbb{P}(L)
\end{tikzcd}.
\end{equation}

Moreover by \cite[Lemma 2.4]{cartwright2011mustafin}, the $m-$th primary component of $\mathcal{M}_{\underline{a}}(\Gamma)_k$ projects onto the $m-$th primary component of $\mathcal{M}_{\underline{a}}(\{L_m,L_n,L_o\})_{k}$. Thus, the ideal of the $l-$th primary component of $\mathcal{M}_{\underline{a}}(\Gamma)_{k}$ contains the ideal
\begin{equation}
\langle x_{1m},x_{2m},x_{1o},x_{2o} \rangle
\end{equation}
for all $m,o=1,\dots,n+1$ and $m\neq l$ and $o \neq l$, i.e. it contains the ideal in \cref{equ:prim}. We observe that this ideal already cuts out a topological space isomorphic to $\mathbb{P}^2_{k}$. An analogous argument shows that the ideals in \cref{equ:sec} correspond to secondary components of $\mathcal{M}_{\underline{a}}(\Gamma)_{k}$. In order to see that these are all the components, we observe that we have produced $n+1+\binom{n+1}{2}=\binom{n+2}{2}$ irreducible components, which by \cite[Theorem 2.3]{cartwright2011mustafin} is the maximal number.
\end{proof}

\begin{definition}
We fix matrix coefficients $\underline{a}\in \Delta(K)$. Let $C\subset \mathbb{P}^2_{K}$ be an irreducible plane curve. We denote by $\mathfrak{C}^{\underline{a}}$ the Mustafin model of $C$ obtained via \cref{equ:mustmod}. We say $\mathfrak{C}^{\underline{a}}$ has \textit{star-like reduction} over $R$ if
\begin{itemize}
\item the special fiber $\mathfrak{C}^{\underline{a}}_k$ decomposes into $n+1$  irreducible components $C_1,\dots,C_{n+1}$;
\item the component $C_i$ is contained in the $i-$th primary component of $\mathcal{M}_{\underline{a}}(\Gamma)_{k}$, and the reduced component $D_i = C_i^{\mathrm{red}}$ is isomorphic to the subscheme of $\mathbb{P}(L)_{k}^{n+1}$ cut out by
\begin{equation}
\label{equ:comp}
\langle x_{1i},\left(x_{1j},x_{2j}\right)_{j=1,\dots,n+1;j\neq i}\rangle,
\end{equation}
which yields $D_i\cong\mathbb{P}^1_k$.
\end{itemize}
\end{definition}

\begin{theorem}
\label{thm:spec} Assume that the residue field $k$ of the discretely valued ground field $K$ is algebraically closed and let $n\ge2$. 
Let $C\subset \mathbb{P}^2_{K}$ be an irreducible plane curve over $K$. For general coefficients $\underline{a} $ in $\Delta(R_K) \cap R_K^{9 (n+1)}$ the model $\mathfrak{C}^{\underline{a}}$  has star-like reduction over $R_K$. 
\end{theorem}

For general $K$ an analogous results only holds after a passing to a finite extension. 
\begin{corollary}
\label{cor:spec} Let $K$ be any discretely valued field with perfect residue field $k$, $n\ge2$, and let 
$C\subset \mathbb{P}^2_{K}$ be an irreducible  plane curve over $K$. After base change with a finite extension $L$ of $K$ the following holds: For general coefficients $\underline{a} $ in $\Delta(R_L) \cap R_L^{9 (n+1)}$ the model $\mathfrak{C}^{\underline{a}}$  of $C_L \subset\mathbb{P}^2_{L}$ has star-like reduction over $R_L$. 
\end{corollary}
\begin{proof} After base changing $C$ with a discretely valued extension $K^\#$ of $K$ with  residue field $\overline{k}$ we can apply \cref{thm:spec} and find a Zariski open subset $U$  of $\overline{k}^{9 (n+1)}$ such that all preimages $\underline{a}$ in $\Delta(R_{K^\#})\cap R_{K^\#}^{9(n+1)}$ lead to models $\mathfrak{C}^{\underline{a}}$ with star-like reduction over $R_{K^\#}$. If $L \subset K^\#$ is a finite extension of $K$ 
with residue field $\ell$ satisfying $U(\ell) \neq \emptyset$, every choice of coefficients $\underline{a} $ in $\Delta(R_L) \cap R_L^{9 (n+1)}$ reducing to a point in $U(\ell)$ has the property that $\mathfrak{C}^{\underline{a}}$ has star-like reduction over $R_L$.
\end{proof}

The next lemma gives a criterion for a model to have star-like reduction.

\begin{lemma}
\label{lem:step1} We assume that the residue field of $K$ is algebraically closed and $n\ge2$. 
Suppose that for a general choice of $\underline{a}$ in $\Delta(R_K) \cap R_K^{9 (n+1)}$, all irreducible components of the special fiber $\mathfrak{C}^{\underline{a}}_{k}$ are contained in primary components of $\mathcal{M}_{\underline{a}}(\Gamma)_{k}$. Then for a general choice of coefficients $\underline{a}$ the model $\mathfrak{C}^{\underline{a}}$ has star-like reduction over $R_K$.
\end{lemma}

\begin{proof} 
Let $f\in K[x_1,x_2,x_3]$ be an irreducible homogeneous polynomial of degree $d$ such that $C = V(f)$.  Without loss of generality, we may assume that $f$ is saturated with respect to $t$. Further, we consider the embedding of $C$ into $\mathbb{P}(L)$ via $g_i^{-1}$. The closure yields a flat proper model of $C$ over $\mathcal{O}_{k}$, which we denote by $\left(\mathfrak{C}^{\underline{a}}\right)^{(i)}$. We see immediately that for general $\underline{a}$ the subscheme $\left(\mathfrak{C}^{\underline{a}}\right)^{(i)}$  of $\mathbb{P}(L)$ is cut out by
\begin{equation}
F=f(a_{11}^{(i)}x_{1}+a_{12}^{(i)}tx_{2}+a_{13}^{(i)}t^2x_{3},a_{21}^{(i)}x_{1}+a_{22}^{(i)}tx_{2}+a_{23}^{(i)}t^2x_{3},a_{31}^{(i)}x_{1}+a_{32}^{(i)}tx_{2}+a_{33}^{(i)}t^2x_{3}).
\end{equation}
Moreover, we compute the reduction of $F$ modulo the valuation ideal $\mathfrak{m}$, which we denote by $\tilde{F}$. Then the special fiber $\left(\mathfrak{C}^{\underline{a}}\right)^{(i)}$ is cut out by $\langle\tilde{F}\rangle$ in $\mathbb{P}^2_{k}=\mathbb{P}(L)_{k}$. We observe that
\begin{equation}
\tilde{F}=h(\overline{a}_{11}^{(i)},\overline{a}_{21}^{(i)},\overline{a}_{31}^{(i)})x_1^d,
\end{equation}
where $h(\overline{a}_{11}^{(i)},\overline{a}_{21}^{(i)},\overline{a}_{31}^{(i)})$ is a non-zero polynomial expression in $\overline{a}_{11}^{(i)},\overline{a}_{21}^{(i)},\overline{a}_{31}^{(i)}$ for generic choices of these coefficients. Therefore, we obtain
\begin{equation}
\label{equ:proj}
\left(\mathfrak{C}^{\underline{a}}\right)^{(i)}_{k}=\mathrm{Proj}\left(\faktor{{k}[x_1,x_2,x_3]}{\langle x_1^d\rangle}\right),
\end{equation}
which yields a degree $d$ subscheme of $\mathbb{P}^2_k$ whose underlying topological space is isomorphic to $\mathbb{P}^1_k$.\par
We now relate this to $\mathfrak{C}^{\underline{a}}$. For this purpose, recall that the Chowring of $\left(\mathbb{P}_K^2\right)^{n+1}$ is given by
\begin{equation}
\mathcal{A}=\faktor{\mathbb{Z}[H_1,\dots,H_{n+1}]}{\langle H_1^3,\dots,H_{n+1}^3\rangle},
\end{equation}
where $H_i$ is the hyperplane class in the $i-$th factor. The Chow class of $\mathfrak{C}^{\underline{a}}_{K}$ is given as $d$ times the sum over all monomials of degree $2(n+1)-1$ in $\mathcal{A}$: The Chow class of $C\subset\mathbb{P}^2_{K}$ is given by $d\cdot H$, where $H$ is the hyperplane class. Under the diagonal embedding the hyperplane class $H$ pushes forward to
\begin{equation}
\sum_{\substack{0\le n_i\le 2\\\sum n_i=2(n+1)-1}}\prod H_i^{n_i}.
\end{equation}

Each monomial in $\mathcal{A}$ is given by $\prod H_i^{n_i}$, where $0\le n_i\le 2$. There are exactly $n+1$ such monomials of degree $2(n+1)-1$, which are given by those $\prod H_i^{n_i}$, such that there exists a $j\in\{1,\dots,n+1\}$ with $n_i=2$ for $i\neq j$ and $n_j=1$. The Chow classes of $\mathfrak{C}^{\underline{a}}_{K}$ and $\mathfrak{C}^{\underline{a}}_{k}$ coincide (see the discussion previous to Corollary 20.3 in \cite{fulton2013intersection}). Moreover, the Chow class of $\mathcal{M}_{\underline{a}}(\Gamma)_k$ is given by the sum over all monomials of degree $2n$.  Then, there are two cases for the monomials $\prod H_i^{n_i}$, where $0\le n_i\le 2$, appearing in the Chow class of $\mathcal{M}_{\underline{a}}(\Gamma)_k$:
\begin{itemize}
\item There exists $j$, such that $n_i=2$ for $i\neq j$ and $n_j=0$. Such a monomial corresponds to the $j-$th primary component, which we denote by $Y_j$.
\item There exists $j,l$, such that $n_{\nu}=2$ for $\nu\neq j,l$ and $n_j=n_l=1$. This monomial corresponds to the secondary component which projects to $\mathbb{P}^1_k$ via $p_j$ and $p_l$, which we denote by $Y_{jl}$.
\end{itemize}
We consider a reduced irreducible component $D$ of $\mathfrak{C}^{\underline{a}}_{k}$, which is contained in the $i-$th primary component of  $\mathcal{M}_{\underline{a}}(\Gamma)_{k}$. Then analogous to the proof of \cref{lem:prime}, it projects onto the single irreducible component of $\left(\mathfrak{C}^{\underline{a}}\right)^{(i)}$. By \cref{equ:proj}, the reduced irreducible component $D$ is cut out by an ideal containing $x_{1i}$ for general coefficients. Further, the ideal of $D$ contains the ideal defining the $i-$th primary component of $\mathcal{M}_{\underline{a}}(\Gamma)_{k}$. Moreover, as proved in \cref{lem:prime}, for general coefficients the $i-$th primary component is cut out by \cref{equ:prim}. This already yields an irreducible component with reduced structure isomorphic to $\mathbb{P}^1_{k}$ and it is the only irreducible component, which lies in the $i-$th primary component of $\mathcal{M}_{\underline{a}}(\Gamma)_{k}$. Note, that this is the ideal given in \cref{equ:comp}.\par
By assumption all components of $\mathfrak{C}^{\underline{a}}_{k}$ lie in primary components. As we have already seen that for general coefficients there lies at most one component of $\mathfrak{C}^{\underline{a}}_{k}$ in each primary component, this assumption yields that there are at most $n+1$ irreducible components of $\mathfrak{C}^{\underline{a}}_{k}$. What is left to prove is that there are exactly $n+1$ irreducible components, i.e. there is one in each primary component of $\mathcal{M}_{\underline{a}}(\Gamma)_k$. To see this, we observe that if there is an irreducible component of $\mathfrak{C}^{\underline{a}}_{k}$ contained in the $i-$th primary component, it contributes the Chow class $M_i=\alpha_i\prod H_j^{n_j}$, where $n_j=2$ for $j\neq i$ and $n_i=1$, and where $\alpha_i$ is the multiplicity of the component. As each monomial is only contributed by a single component, we observe that $\alpha_i=d$. Thus, there are exactly $n+1$ components with multiplicity $d$. 
\end{proof}

We need the following geometric lemma for the proof of our next result.

\begin{lemma} 
\label{lemma-fibers}
Let $S$ be a noetherian irreducible and reduced scheme with generic point $\eta$ and consider an $S$-scheme $X$ of finite type. Let $Y$ be an irreducible
closed subset of $X$ with non-empty generic fiber $Y_\eta$, and let $Z$ be any closed subset of $X$. 

i) If the generic fiber  $Y_\eta$ is not contained in $Z$, then there exists a dense open subset $U$ of $S$ such that for all $s \in U$ the fiber $Y_s$ is not contained in $Z$. 

ii) If  the generic fiber $Y_\eta$ is contained in $Z$, then there exists a dense open subset $U$ of $S$ such that for all $s \in U$ the fiber $Y_s$ is non-empty and contained in $Z$.

\end{lemma}
\begin{proof}

i) The subset $Y \backslash Z$ is open and non-empty in $Y$, which we endow with the reduced scheme structure. By generic flatness \cite[28.26.1]{SP}, there exists an open dense subset $V$ of $S$, such that $Y_V$ is flat over $V$ and hence open. Therefore the image $U$ of the open subset $(Y \backslash Z) \cap Y_V$ is open in $V \subset S$. It is non-empty, since it contains the generic point, and therefore dense. For all $s \in U$ we find that $Y_s$ is not contained in $Z$. 

ii) A similar argument as in i) using generic flatness implies that there exists an open subset $U$ of $S$ with $Y_s$ non-empty for all $s \in U$. Then the claim follows since the closure of $Y_\eta$ is $Y$.
\end{proof}

Our next goal is to show that for sufficiently general choices of $a_{ij}^{(l)}$, the irreducible components of $\mathfrak{C}^{\underline{a}}_{k}$ are in fact all contained in primary components.
In  order to show this, we work in the following algebraic set-up.\par 
We consider the ring $\mathcal{R}={R}\left[(A_{ij}^{(l)})_{i,j=1,\dots,3;l=1,\dots,n+1}\right]$ and the field $\mathcal{K}=\mathrm{Quot}(\mathcal{R})$. Further, we consider a $\mathcal{K}-$vector space of dimension $3$, which we denote by $\mathcal{V}$ with standard basis $\mathfrak{e}_1,\mathfrak{e}_2,\mathfrak{e}_3$. We denote by $\mathbb{P}(\mathcal{V})$ the projective space, and by  $\mathcal{L}=\mathcal{R}\mathfrak{e}_1+\mathcal{R}\mathfrak{e}_2+\mathcal{R}\mathfrak{e}_3$ the standard lattice. Let $\mathbb{P}(\mathcal{L})=\mathrm{ProjSym}(\mathrm{Hom}_{\mathcal{R}}(\mathcal{L},\mathcal{R}))$. We further consider the matrices
\begin{equation}
\mathfrak{g}_i=\begin{pmatrix}
A_{11}^{(i)} & A_{12}^{(i)} & A_{13}^{(i)}\\
A_{21}^{(i)} & A_{22}^{(i)} & A_{23}^{(i)}\\
A_{31}^{(i)} & A_{32}^{(i)} & A_{33}^{(i)}
\end{pmatrix}\cdot
\begin{pmatrix}
1 & 0 & 0\\
0 & t & 0\\
0 & 0 & t^2
\end{pmatrix}
\end{equation}
and the morphism of $\mathcal{R}$-schemes
\begin{equation}
\label{equ:mustemb}
\mathbb{P}(\mathcal{V})\xrightarrow{(\mathfrak{g}_1^{-1}\times\dots\times\mathfrak{g}_{n+1}^{-1})\circ\Delta}\mathbb{P}(\mathcal{L})^{n+1}
\end{equation}
and denote the closure of this map endowed with the reduced scheme structure by $\mathcal{N}(\Gamma')$, where we put $\Gamma'=\{\mathfrak{g}_1\mathcal{L},\dots,\mathfrak{g}_{n+1}\mathcal{L}\}$. Let $y_1,y_2,y_3$ the standard coordinates of $\mathbb{P}(\mathcal{L})$ and consider the projection
\begin{equation}
\mathfrak{p}_i:\mathbb{P}(\mathcal{L})^{n+1}\to\mathbb{P}(\mathcal{L})
\end{equation}
to the $i-$th factor. We then denote $y_{ij}=p_j^\ast y_i$. Similar to our discussion in \cref{sec:must}, we observe that $\mathcal{N}(\Gamma')$ is the subscheme of $\mathbb{P}(\mathcal{L})^{n+1}$ cut out by
\begin{equation}
I_2\begin{pmatrix}
\mathfrak{g}_1\begin{pmatrix}
y_{11}\\
y_{21}\\
y_{31}
\end{pmatrix} &
\cdots
& \mathfrak{g}_n\begin{pmatrix}
y_{1n+1}\\
y_{2n+1}\\
y_{3n+1}
\end{pmatrix}
\end{pmatrix}\cap \mathcal{R}[(y_{ij})].
\end{equation}
We consider an irreducible homogeneous polynomial in $f'\in\mathcal{K}[y_1,y_2,y_3]$. This defines a closed subscheme $\mathcal{C}\subset\mathbb{P}(\mathcal{V})$ over $\mathcal{K}$.  We obtain a scheme $\mathscr{C}$ over $\mathrm{Spec}(\mathcal{R})$ with generic fiber $\mathcal{C}$ by embedding $\mathcal{C}$ via \cref{equ:mustemb} into $\mathcal{N}(\Gamma')$ and taking the closure endowed with the reduced induced structure.\par 
We now consider the ring $\mathcal{R}'=\faktor{\mathcal{R}}{(t)}={k}[(A_{ij}^{(l)})_{i,j=1,2,3;l=1,\dots,n+1}]$. We denote the pullbacks to $\mathrm{Spec}(\mathcal{R}')$ by $\mathcal{N}'(\Gamma')=\mathcal{N}(\Gamma')_{\mathrm{Spec}(\mathcal{R}')}$ and $\mathscr{C}'=\mathscr{C}_{\mathrm{Spec}(\mathcal{R}')}$. We note that $\mathcal{N}(\Gamma')$ and $\mathscr{C}$ are of finite type over the noetherian ring $\mathcal{R}$ and therefore $\mathcal{N}'(\Gamma')$ and $\mathscr{C}'$ are of finite type over $\mathcal{R}'$. \par 

For every choice of coefficients $\underline{a}$ in $\Delta(R) \cap R^{9 (n+1)}$ we have a natural homomorphism $\lambda_{\underline{a}}: \mathcal{R} \rightarrow R$ mapping $A_{ij}^{(l)}$ to $a_{ij}^{(l)}$. The corresponding base change of  $\mathcal{N}(\Gamma')$ with $R$ is by construction isomorphic to $\mathcal{M}_{\underline{a}}(\Gamma)$, where $\Gamma = \lambda_{\underline{a}} (\Gamma')$ is the set of lattices we get by inserting the coefficients $a_{ij}^{(l)}$ for $A_{ij}^{(l)}$. Its special fiber $\mathcal{M}_{\underline{a}}(\Gamma)_k$ is therefore isomorphic to the base change of $\mathcal{N}'(\Gamma')$ with respect to the homomorphism $\overline{\lambda}_{\underline{a}}: \mathcal{R}' \rightarrow k$ sending $A_{ij}^{(l)}$ to the elements $\overline{a}_{ij}^{(l)}$ in the residue field.\par 
Similarly, the base change of $\mathscr{C}$ along $\lambda_{\underline{a}}$ yields a scheme $\lambda_{\underline{a}}^{\ast}\mathscr{C}$.  We consider the map $\Lambda_{\underline{a}}$ obtained by composing $\lambda_{\underline{a}}$ with $R\hookrightarrow K$. Assume that  $\Lambda_{\underline{a}}^{\ast}  \mathscr{C} \subset \mathbb{P}^2_K$ is a reduced and irreducible curve for general $\underline{a}$. \par
There exists a non-empty open  $V$ such that $\mathscr{C}_V$ is flat  over $V$ \cite[28.26.1]{SP}. If  the residue field $k$ is big enough, so that $V(k) \neq \emptyset$, we find that $\mathfrak{D}^{\underline{a}}:=\lambda_{\underline{a}}^{\ast}\mathscr{C}$ is the Mustafin model of its generic fiber $\mathfrak{D}^{\underline{a}}_K=\Lambda_{\underline{a}}^{\ast}\mathscr{C}$
for general $\underline{a}$.


Before stating the next lemma, we introduce some more notation. Let $\sigma$ be a permutation acting on $\{1,\dots,n+1\}$, then we define the ring isomorphism
\begin{equation}
\tau_{\sigma}:\mathcal{K}[y_1,y_2,y_3]\to \mathcal{K}[y_1,y_2,y_3]
\end{equation}
by $A_{ij}^{(l)}\mapsto A_{ij}^{(\sigma(l))}$ and the homomorphism
\begin{equation}
\label{equ:morperm}
\pi_{\sigma}:\mathcal{R}[\left(y_{\mu\nu}\right)_{\mu=1,2,3;\nu=1,\dots,n+1}]\to \mathcal{R}[\left(y_{\mu\nu}\right)_{\mu=1,2,3;\nu=1,\dots,n+1}]
\end{equation}
by $(A_{ij}^{(l)},y_{il})\mapsto(A_{ij}^{(\sigma(l))},y_{i\sigma(l)})$. 

\begin{lemma}
\label{lem:step2} Assume that the residue field $k$ is algebraically closed. 
Let $\mathcal{C}\subset\mathbb{P}(\mathcal{V})$ be a curve defined by an irreducible homogeneous polynomial $f'\in\mathcal{K}[y_1,y_2,y_3]$ and $n\ge2$.  Furthermore, let $\mathscr{C}$ be the closure of the image of $\mathcal{C}$ under the image of the map in \cref{equ:mustemb} endowed with the reduced structure. Assume that $\mathfrak{D}^{\underline{a}}_K = \Lambda_{\underline{a}}^{\ast}  \mathscr{C} $ is a reduced and irreducible curve for general $\underline{a}$. Then, we have
\begin{enumerate}
\item If $\tau_{\sigma}(I(\mathcal{C}))=I(\mathcal{C})$, then $\pi_{\sigma}(I(\mathscr{C}))=I(\mathscr{C})$. 
\item If $\pi_{\sigma}(I(\mathscr{C}))=I(\mathscr{C})$, there exists an open subset $W\subset \mathrm{Spec}(\mathcal{R}')$, such that for each point $y \in W(k) \subset k^{9 (n+1)}$and each lift $\underline{a} \in \Delta(R) \cup R^{9 (n+1)}$ of $y$ all irreducible components of $\mathscr{C}'_y$ are contained in primary components of the special fiber $\mathcal{N}'(\Gamma')_y\cong\mathcal{M}_{\underline{a}}(\Gamma)_k$ of the corresponding Mustafin variety.
\end{enumerate}
\end{lemma}

\begin{proof}
The first part of the lemma is straightforward. For the second part, let $\eta$ be the generic point of $\mathrm{Spec}(\mathcal{R}')$. We study the irreducible components of the generic fibers $\mathcal{N}'(\Gamma')_{\eta}$ and $\mathscr{C}'_{\eta}$.
(1)  By the same computations as in the proof of \cref{lem:prime}, we obtain that $\mathcal{N}'(\Gamma')_{\eta}$ is contained in
\begin{equation}
\mathcal{Y}_1\cup\dots\cup\mathcal{Y}_{n+1}\cup\bigcup_{i\neq l}\mathcal{Y}_{il},
\end{equation}
where the scheme $\mathcal{Y}_i$ is cut out by
 \begin{equation}
\langle \left(y_{1j},y_{2j}\right)_{j=1,\dots,n+1;j\neq i}\rangle
\end{equation}
and the scheme $\mathcal{Y}_{il}$ is cut out by
\begin{equation}
\langle y_{1l},y_{1i},\left(y_{1j},y_{2j}\right)_{j=1,\dots,n+1;j\neq i,l}\rangle
\end{equation}
for $i,l=1,\dots,n+1$ and $i\neq l$. We claim that 
\begin{equation}
\mathcal{N}'(\Gamma')_{\eta}=\mathcal{Y}_1\cup\dots\cup\mathcal{Y}_{n+1}\cup\bigcup_{i\neq l}\mathcal{Y}_{il}
\end{equation}
is a decomposition into irreducible components.

In fact, assume that there exists an irreducible component $\cX$ of $\mathcal{N}'(\Gamma')_{\eta}$ which is not contained in the union on the right hand side. 
Then by \cref{lemma-fibers} there exists a non-empty open set $W\subset\mathrm{Spec}(\mathcal{R'})$, such that for all $y\in W$ we have 
\begin{equation}
\mathcal{N}'(\Gamma')_y\subsetneq\left(\overline{\mathcal{Y}_1}\right)_y\cup\dots\cup\left(\overline{\mathcal{Y}_{n+1}}\right)_y\cup\bigcup_{i\neq l}\left(\overline{\mathcal{Y}_{il}}\right)_y,
\end{equation}
which contradicts \cref{lem:prime}, since for all points $y\in \Delta(k) \cap W(k)$  the scheme $\mathcal{N}'(\Gamma')_y$ is the special fiber of a Mustafin variety. Moreover, we see immediately that for all  points $y\in \Delta(k) \cap W(k)$, the subset $\left(\overline{\mathcal{Y}}_i\right)_y$ is the $i-$th primary component of the respective Mustafin variety and $\left(\overline{\mathcal{Y}}_{ij}\right)_y$ is a secondary component mapping onto $\mathbb{P}^1$ via the projections to the $i-$th and $j-$th factor.\par

(2)  Let $\mathcal{Z}$ be an irreducible component of $\mathscr{C}'_{\eta}$, hence $\mathcal{Z} \subset\mathcal{N}'(\Gamma')_\eta$. We claim that there exists an index $i$, such that $\mathcal{Z}\subset\mathcal{Y}_i$. We first treat the case $n >2d$. Assume there exists no such $i$, then there exist $i \neq l$, such that $\mathcal{Z}\subset \mathcal{Y}_{il}$ and $\mathcal{Z}\not\subset \mathcal{Y}_i,\mathcal{Y}_l$. 
\par 
Let $I_1=I(\mathcal{N}(\Gamma'))$ and $I_2=I(\mathscr{C})$ be the multihomogeneous ideals in $\mathcal{R}[\left(y_{\mu\nu}\right)_{\substack{\mu=1,2,3;\\\nu=1,\dots,n+1}}]$. Let $\sigma$ be a permutation on $\{1,\dots,n+1\}$ and recall the ring isomorphism
\begin{equation}
\label{equ:morperm}
\pi_{\sigma}:\mathcal{R}[\left(y_{\mu\nu}\right)_{\mu=1,2,3;\nu=1,\dots,n+1}]\to \mathcal{R}[\left(y_{\mu\nu}\right)_{\mu=1,2,3;\nu=1,\dots,n+1}]
\end{equation}
induced by $(A_{ij}^{(l)},y_{il})\mapsto(A_{ij}^{(\sigma(l))},y_{i\sigma(l)})$. By construction, we have $\pi_{\sigma}(I_1)=I_1$ and by assumption, we have $\pi_{\sigma}(I_2)=I_2$. Moreover, let $J_1=I(\mathcal{N}'(\Gamma')_\eta)$ and $J_2=I(\mathscr{C}'_\eta)$ be the corresponding multihomogeneous ideals  in $\mathcal{Q}=\mathrm{Quot}(\mathcal{R}')[\left(y_{\mu\nu}\right)_{\mu=1,2,3;\nu=1,\dots,n+1}]$. The isomorphism $\pi_{\sigma}$ induces an isomorphism
\begin{equation}
\pi_{\sigma}':\mathcal{Q}\to\mathcal{Q}
\end{equation}
with $\pi_{\sigma}'(J_1)=J_1$ and $\pi_{\sigma}'(J_2)=J_2$. The induced isomorphism of schemes $\Pi_{\sigma}':\mathcal{N}'(\Gamma')_{\eta}\to\mathcal{N}'(\Gamma')_{\eta}$ yields $\Pi_{\sigma}'(\mathcal{Y}_i)=\mathcal{Y}_{\sigma(i)}$ and $\Pi_{\sigma}'(\mathcal{Y}_{ij})=\mathcal{Y}_{\sigma(i)\sigma(j)}$. Moreover as $\pi_{\sigma}'(J_2)=J_2$, we also have an isomorphism of schemes $\restr{\Pi_{\sigma}'}{\mathscr{C}_{\eta}'}:\mathscr{C}_{\eta}'\to\mathscr{C}_{\eta}'$. Let $\mathcal{Z}\subset \mathcal{Y}_{ij}$ and $\mathcal{Z}\not\subset\mathcal{Y}_i,\mathcal{Y}_j$. For any permutation $\sigma$, we obtain that $\restr{\Pi_{\sigma}'}{\mathscr{C}_\eta'}(\mathcal{Z})\subset\mathcal{Y}_{\sigma(i)\sigma(j)}$ and $\restr{\Pi_{\sigma}'}{\mathscr{C}_\eta'}(\mathcal{Z})\not\subset \mathcal{Y}_{\sigma(i)},\mathcal{Y}_{\sigma(j)}$. We first prove the following claim.

If $\sigma$ does not stabilize $\{i,j\}$, then it follows that $\restr{\Pi_{\sigma}'}{\mathscr{C}_\eta'}(\mathcal{Z})\neq\mathcal{Z}$.

In fact,  if $ \restr{\Pi_{\sigma}'}{\mathscr{C}_\eta}(\mathcal{Z})  =\mathcal{Z}$, we find $\mathcal{Z}\subset\mathcal{Y}_{ij}\cap\mathcal{Y}_{\sigma(i)\sigma(j)}$. However, by considering the respective ideals, we have
\begin{itemize}
\item $\mathcal{Y}_{ij}\cap\mathcal{Y}_{\sigma(i)\sigma(j)}\subset\mathcal{Y}_i$, if $i=\sigma(i)$ or $i=\sigma(j)$
\item $\mathcal{Y}_{ij}\cap\mathcal{Y}_{\sigma(i)\sigma(j)}\subset\mathcal{Y}_j$, if $j=\sigma(j)$ or $j=\sigma(i)$
\end{itemize}
Hence we may assume that $\{i,j\}$ is disjoint from $\{\sigma(i), \sigma(j)\}$. But in this case the intersection $\mathcal{Y}_{ij}\cap\mathcal{Y}_{\sigma(i)\sigma(j)}$ is a point, that is contained in all primary components, which violates again the condition that $\mathcal{Z}$ is not contained in any $\mathcal{Y}_i$. This proves our claim.

After passing to a smaller subset $W$ if necessary, we find by \cref{lemma-fibers} that for all  $y\in \Delta(k) \cap W(k)$ the closed subsets  $\left(\left(\overline{\mathcal{Z}_{\mu\nu}}\right)_y\right)_{\mu,\nu}$  of $\mathscr{C}_y'$ are pairwise different.\par
As already observed before stating the lemma, after possible shrinking $W$, we may assume that for all $y\in \Delta(k)\cap W(k)$, there exists a lift $\underline{a}$ of $y$, such that $\mathfrak{D}^{\underline{a}}$ is the Mustafin model of its generic fiber with special fiber $\mathfrak{D}^{\underline{a}}_k\cong\mathscr{C}'_y$.\par
Now we want to see that after possible shrinking $W$ again, we have that $\left(\left(\overline{\mathcal{Z}_{\mu\nu}}\right)_y\right)_{\mu,\nu}$ contain pairwise different irreducible components of $\mathscr{C}'_y$ for $y\in W$. In order to see this, we first consider the following decomposition into irreducible components
\begin{equation}
\mathscr{C}_{\eta}'=\bigcup_{\mu,\nu} Z_{\mu\nu}\cup\bigcup_{\gamma} \Theta_{\gamma}.
\end{equation}
We observe
\begin{equation}
\mathcal{Z}_{ij}\not\subset\bigcup_{\substack{\mu,\nu\\\{\mu,\nu\}\neq\{i,j\}}}\mathcal{Z}_{\mu\nu}\cup\bigcup_{\gamma}\Theta_{\gamma}
\end{equation}
and thus after possibly shrinking $W$, we have by \cref{lemma-fibers} that
\begin{equation}
\left(\overline{\mathcal{Z}_{ij}}\right)_y\not\subset\bigcup_{\substack{\mu,\nu\\\{\mu,\nu\}\neq\{i,j\}}} \left(\overline{\mathcal{Z}_{\mu\nu}}\right)_y\cup\bigcup_{\gamma}\left(\overline{\Theta_{\gamma}}\right)_y
\end{equation}
and by \cite[36.22.5]{SP} that
\begin{equation}
\mathscr{C}'_y=\bigcup_{\mu,\nu} \left(\overline{\mathcal{Z}_{\mu\nu}}\right)_y\cup\bigcup_{\gamma}\left(\overline{\Theta_{\gamma}}\right)_y
\end{equation}
for all $y\in \Delta(k)\cap W(k)$. Therefore, $\left(\overline{\mathcal{Z}_{ij}}\right)_y$ contains an irreducible component of $\mathscr{C}'_y$, which is not contained in
\begin{equation}
\bigcup_{\substack{\mu,\nu\\\{\mu,\nu\}\neq\{i,j\}}} \left(\overline{\mathcal{Z}_{\mu\nu}}\right)_y\cup\bigcup_{\gamma}\left(\overline{\Omega_{\gamma}}\right)_y
\end{equation}
for all $y\in \Delta(k)\cap W(k)$. By symmetry, after possibly shrinking $W$ again, we have that
\begin{equation}
\left(\left(\overline{\mathcal{Z}_{\mu\nu}}\right)_y\right)_{\mu,\nu}
\end{equation}

contain pairwise different irreducible components of $\mathscr{C}'_y$ for $y\in \Delta(k)\cap W(k)$. For each $\left(\overline{\mathcal{Z}_{ij}}\right)_y$ we pick such a component, which we denote by $\mathscr{Z}_{ij}^y$.
Finally, all irreducible components of $\mathscr{C}'_y$ are $1-$dimensional and therefore $\mathrm{dim}\left(\mathscr{Z}_{ij}^y\right)=1$ for all $y\in \Delta(k)\cap W(k)$.\\
The above considerations show that for general $\underline{a}$, the special fiber of $\mathfrak{D}^{\underline{a}}$ contains at least $\frac{n (n+1)}{2}$ irreducible components. We now prove that this yields a contradiction, which finishes this step of the proof. Recall that we assume $n>2d$.

\begin{itemize}
\item First, for $y\in\Delta(k)\cap W(k)$, we consider the map
\begin{equation}
\omega:\{(i,j)\in\{1,\dots,n+1\}^2\mid i<j,i\neq j\}\to\{1,\dots,n+1\}
\end{equation}
given by
\begin{align}
\omega((i,j))=\begin{cases}
i,\,\textrm{if}\quad p_i\left(\mathscr{Z}_{ij}^y\right)\cong\mathbb{P}^1\\
j,\,\textrm{if}\quad p_i\left(\mathscr{Z}_{ij}^y\right)\cong pt
\end{cases}
\end{align}
Then, we have
\begin{equation}
\bigcup \omega^{-1}(i)=\{(i,j)\in\{1,\dots,n+1\}^2\mid i<j,i\neq j\}
\end{equation}
and thus
\begin{equation}
\sum_i \left|\omega^{-1}(i)\right|\ge \left|\{(i,j)\in\{1,\dots,n+1\}^2\mid i<j,i\neq j\}\right|=\binom{n+1}{2}=\frac{n(n+1)}{2}.
\end{equation}
By assumption we have $n(n+1)>2d^2$, hence there exists $i\in\{1,\dots,n+1\}$ with $\left|\omega^{-1}(i)\right|>d$ as otherwise $\frac{n(n+1)}{2}\le\sum \left|\omega^{-1}(i)\right|\le (n+1) \cdot d<\frac{n(n+1)}{2}$, which is a contradiction. 
\item We first observe that for all $y\in \Delta(k)\cap W(k)$, we have $\mathscr{Z}_{ij}^y\subset\left(\overline{Y_{ij}}\right)_y$, $\mathrm{dim}\mathscr{Z}_{ij}^y=1$ with $\left(\overline{Y_{ij}}\right)_y$ being the secondary component of the respective Mustafin variety projecting to $\mathbb{P}^1$ via $p_i$ and $p_j$ and to a point via $p_l$ ($l\neq i,j$). Thus, we have $p_i\left(\mathscr{Z}_{ij}^y\right)\cong\mathbb{P}^1$ or $p_j\left(\mathscr{Z}_{ij}^y\right)\cong\mathbb{P}^1$.\par
Secondly, we observe that the Chow class of $\left(\mathscr{Z}_{ij}^y\right)$ is given by 
\begin{equation}
\alpha^{ij}H_i\prod_{l\neq i}H_l^2+\beta^{ij}H_j\prod_{l\neq j}H_l^2
\end{equation}
with $\alpha^{ij},\beta^{ij}\in\mathbb{Z}_{\ge0}$, since it is a component of a curve inside a secondary component projecting to $\mathbb{P}^1_k$ via the projections $p_i$ and $p_j$. Moreover, we note that if $\omega((i,j))=i$ (resp. $\omega((j,i))=j$), then we have $\alpha^{ij}\neq 0$. Now, we choose $i$, such that $\left|\omega^{-1}(i)\right|> d$ and denote by $\Omega_i$ the set of all $j$, such that $j\neq i$ and either $(i,j)\in\omega^{-1}(i)$ or $(j,i)\in\omega^{-1}(i)$. Therefore, there are at least $d+1$ irreducible distinct subsets $\left(\mathscr{Z}_{ij}^y\right)$, such that  $p_i\left(\mathscr{Z}_{ij}^y\right)\cong\mathbb{P}^1$.
Then we see that the Chow class of
\begin{equation}
\bigcup_{\substack{i<j\\p_i\left(\mathscr{Z}_{ij}^y\right)\cong\mathbb{P}^1}}\left(\mathscr{Z}_{ij}^y\right)\subset \mathfrak{C}^{\underline{a}}_k\quad \textrm{(for $\underline{a}$ a lift of $y$)}
\end{equation}
contains
\begin{equation}
\sum_{j\in\Omega_i}\alpha^{ij}H_i\prod_{l\neq i}H_l^2.
\end{equation}
However, we see that
\begin{equation}
\sum_{j\in\Omega_i}\alpha^{ij}\ge|\Omega_i|\ge d+1
\end{equation}
as each $\alpha^{ij}$ in the sum is a positive. This is a contradiction to the fact that the Chow class of $\mathfrak{C}^{\underline{a}}_k$ is $d$ times the sum over all monomials of degree $2(n+1)-1$.
\end{itemize}

Thus, we have obtained a contradition to our assumption and conclude that there exists some $i\in\{1,\dots,n+1\}$, such that $\mathcal{Z}\subset\mathcal{Y}_i$. \par

Thus, it follows from \cref{lemma-fibers} that there exists an open subset $W\subset\mathrm{Spec}(\mathcal{R}')$, such that for each  $y\in W(k)$ all irreducible components of $\mathscr{C}_y'$ are contained in primary components of $\mathcal{N}'(\Gamma')_y$. 

(3) We have proved the result for $n>2d$ in (2). In order to deduce the result for $n\le 2d$, we consider $m$ with $m>2d$. For an $m+1-$tuple $\underline{a}\in\Delta(R)$ of coefficients, we denote by $\underline{a}_0$ its projection onto the first $n+1$ entries. Let $g_1,\dots,g_{m+1}$ be matrices given by $\underline{a}$, such that $g_1,\dots,g_{n+1}$ are given by $\underline{a}_0$. Thus, we obtain Mustafin varieties $\mathcal{M}_{\underline{a_0}}(\{L_1,\dots,L_{n+1}\})$ and $\mathcal{M}_{\underline{a}}(\{L_1,\dots,L_{m+1}\})$. By embedding $C$ into the respective Mustafin varieties and taking the closure, we obtain flat models $\mathfrak{C}^{\underline{a}_0}\subset\mathcal{M}_{\underline{a}_0}(\{L_1,\dots,L_{n+1}\})$ and $\mathfrak{C}^{\underline{a}}\subset\mathcal{M}_{\underline{a}}(\{L_1,\dots,L_{m+1}\})$ such that  the projection on the first $n+1$ factors of $\mathbb{P}(L)^{m+1}$ maps $\mathfrak{C}^{\underline{a}}_{k}\subset\mathcal{M}_{\underline{a}}(\{L_1,\dots,L_{m+1}\})_{k}$  onto $\mathfrak{C}^{\underline{a}_0}_{k}\subset\mathcal{M}_{\underline{a}_0}(\{L_1,\dots,L_{n+1}\})_{k}$.\par
We now consider a general choice of $\underline{a}$ and therefore of $\underline{a}_0$. More precisely,  for $m$ we fix the open dense subset $W$ obtained in (2), such that for each $y\in W(k)$ all irreducible components of $\mathscr{C}_y'$ are contained in primary components of $\mathcal{N}'(\Gamma')_y$. For $y \in W(k)$, we denote the projection to the coordinates $A_{ij}^{(l)}$ with $l \leq n+1$ by $y_0$. Choose a lift $\underline{a}$ of $y$ with projection $\underline{a}_0$ to the coordinates $a_{ij}^{(l)}$ for $l \leq n+1$, i.e. $\underline{a}_0$ is a lift of $y_0$. Then each component of $\mathfrak{C}^{\underline{a}}_{k}\,(\cong \mathscr{C}'_y)$ is contained in a primary component. The projection onto the first $n+1$ factors either maps such  a component to a point or to an irreducible component, which will also be contained in a primary component of  $\mathcal{M}_{\underline{a}_0}(\{L_1,\dots,L_{n+1}\})_{k}$. Thus, the lemma follows.
\end{proof}

We are finally ready to proof \cref{thm:spec}.

\begin{proof}[Proof of {\cref{thm:spec}}]
We consider the irreducible homogeneous polynomial $f$ defining $C\subset\mathbb{P}^2_{K}$. As irreducibility of polynomials is preserved under purely transcendental field extension, it induces a homogeneous irreducible polynomial in $\mathcal{K}[y_1,y_2,y_3]$. This defines a subscheme $\mathcal{C}\subset\mathbb{P}(\mathcal{V})$ and we obtain a scheme $\mathscr{C}$ over $\mathrm{Spec}(\mathcal{R})$ with generic fiber $\mathcal{C}$ by embedding $\mathcal{C}$ into $\mathcal{N}(\Gamma')$ and taking the closure. Then, we have $\Lambda_{\underline{a}}^{\ast}\mathscr{C}\cong C$ for all $\underline{a} \in \Delta(R)$, and thus the condition in \cref{lem:step2} is satisfied. By construction, we have $\tau_{\sigma}(I(\mathcal{C}))=I(\mathcal{C})$. Then, it follows from \cref{lem:step2} that for a general choice of $\underline{a}$ all components of $\mathfrak{C}^{\underline{a}}_{k}$ lie in the primary components of $\mathcal{M}_{\underline{a}}(\Gamma)_k$. Hence the theorem follows from \cref{lem:step1}. 
\end{proof}

\section{Models of syzygy bundles}
\label{sec:plane}
After studying degenerations of plane curves in the preceeding section, we will now study degenerations of syzygy bundles on curves.

Recall that $K$ is a discretely valued field with ring of integers $R$ and perfect residue field $k$.\par 
We fix two positive integers $n\geq 2$ and $\rho$ and  non-negative numbers $d_1,\dots,d_{n+1}\le \rho$ with $\sum_{j = 1}^{n+1} d_j=n\rho$. Let $F_1,\dots,F_{n+1}$ be polynomials with
\begin{equation}
F_i\in\mathrm{Sym}_{j\neq i}R[x_{1j},x_{2j},x_{3j}]^{(\rho-d_j)},
\end{equation}
where $R[x_{1j},x_{2j},x_{3j}]^{(\rho-d_j)}$ denotes the $R$-submodule of the polynomial ring consisting of all homogenous polynomials of degree $\rho - d_j$. Hence $F_i$ is a linear combinattion of products of $n$ homogenous polynomials, each in a different set of variables. 

We define a coherent sheaf $\mathcal{E}'$ on $\mathcal{M}_{\underline{a}}(\Gamma)$ by the exact sequence
\begin{equation}
\label{equ:cohsheaf1}
0\to \mathcal{E}'\to \bigoplus_{j=1}^{n+1}p_j^\ast\mathcal{O}_{\mathbb{P}(L)}(\rho-d_j)\xrightarrow{(F_1,\dots,F_{n+1})} \bigotimes_{j=1}^{n+1}p_j^\ast\mathcal{O}_{\mathbb{P}(L)}(\rho-d_j),
\end{equation}
which is well-defined since
\begin{align}
F_i\in \bigotimes_{j\neq i}\Gamma(\mathcal{M}_{\underline{a}}(\Gamma),p_j^\ast\mathcal{O}_{\mathbb{P}(L)}(\rho-d_j)).
\end{align}

Then $\cE'$ is a coherent sheaf on $\mathcal{M}_{\underline{a}}(\Gamma)$. For any choice of coefficients $\left(c_{ij}^{(l)}\right)_{\substack{i,j=1,2,3;\\l=1,\dots,n+1}}\in R^{9(n+1)} \cap \Delta(R)$  we consider the linear forms
\begin{equation}
z_{il}=c_{i1}^{(l)}x_{1l}+c_{i2}^{(l)}x_{2l}+c_{i3}^{(l)}x_{3l}.
\end{equation}
Then $z_{il} = p_l^\ast ( z_i^{(l)})$ for $z_{i}^{(l)}=c_{i1}^{(l)}x_{1}+c_{i2}^{(l)}x_{2}+c_{i3}^{(l)}x_{3}$.
 Let $D(z_{il})=p_l^{-1}(D_+(z_i^{(l)}))$ and define the open subset
\begin{equation}
U'_{\underline{c}}=\bigcap_{i,j} D(z_{ij})
\end{equation}
of $\mathcal{M}_{\underline{a}}(\Gamma)$. Then there exist $\beta_{ij}\in\mathcal{O}^{\times}(U'_{\underline{c}})$, such that $\beta_{ij}z_{1j}=z_{ij}$ on $U'_{\underline{c}}$. Thus, we have
\begin{equation}
\restr{p_j^\ast\mathcal{O}(\rho-d_j)}{U'_{\underline{c}}}=z_{1j}^{\rho-d_j}\restr{\mathcal{O}}{U'_{\underline{c}}}.
\end{equation}
Then we can express $F_i$ on $U'_{\underline{c}}$ as
\begin{equation}
F_i=H_i\prod_{\substack{j=1\\j\neq i}}^{n+1}z_{1j}^{\rho-d_j},
\end{equation}
where $H_i=H_i((\beta_{lj})_{lj})\in\mathcal{O}(U'_{\underline{c}})$ is a polynomial in the $(\beta_{lj})_{lj}$. Note that all $H_i\not\equiv0$ because $F_i\not\equiv 0$. Hence
\begin{equation}
U''_{\underline{c},i}=\{z\in\mathcal{M}_{\underline{a}}(\Gamma)\mid (H_i)_z\in\mathcal{O}^{\times}_{\mathcal{M}_{\underline{a}}(\Gamma),z}\}
\end{equation}
is a non-empty open subset of $\mathcal{M}_{\underline{a}}(\Gamma)$. We put 
\begin{equation}
U_{\underline{c},i}=U'_{\underline{c}}\cap U''_{\underline{c},i}.
\end{equation}
We will now show that $\cE'$ is trivial on the non-empty open set $U_{\underline{c},i}\subset\mathcal{M}_{\underline{a}}(\Gamma)$.
\begin{lemma}
\label{lem:free}
For $\underline{c}$ and $U_{\underline{c},i}$ defined as above, we have the following:
\begin{enumerate}
\item The map $\varphi:\bigoplus_{j=1}^{n+1}\restr{p_j^\ast\mathcal{O}(\rho-d_j)}{U_{\underline{c},i}}\xrightarrow{(F_1,\dots,F_{n+1})}\bigotimes_{j=1}^{n+1}\restr{p_j^\ast\mathcal{O}(\rho-d_j)}{U_{\underline{c},i}}$ is surjective.
\item The sheaf $\cE'$ is trivial on $U_{\underline{c},i}$, i.e. we have $\restr{\cE'}{U_{\underline{c},i}}\cong\restr{\mathcal{O}^n}{U_{\underline{c},i}}$.
\end{enumerate}
\end{lemma}

\begin{proof}
Since on $U_{\underline{c},i}$ the morphism $F_i: \restr{p_i^\ast\mathcal{O}(\rho-d_i)}{U_{\underline{c},i}} \rightarrow \bigotimes_{j=1}^{n+1}\restr{p_j^\ast\mathcal{O}(\rho-d_j)}{U_{\underline{c},i}}$ is surjective, we find that $\varphi$ is surjective. 
Note that we have $ \restr{p_j^\ast\mathcal{O}(\rho-d_j)}{U_{\underline{c},i}} \simeq z_{1 j}^{\rho- d_j} \restr{\mathcal{O}}{U_{\underline{c},i}}$,
hence the second claim follows easily.
\end{proof}

For a fixed choice 
\begin{equation}
\underline{a}=\left(a_{ij}^{(l)}\right)_{i,j=1,2,3;l=1,\dots,n+1}\in \Delta(K) \cap R^{9(n+1)}
\end{equation}
and a plane curve $C\subset\mathbb{P}^2_{K}$, let $\mathfrak{C}^{\underline{a}}$ be the associated Mustafin model of $C$ as defined by  \cref{equ:mustmod}.

Now, we fix $d_i$ as above 
and consider the $R$-linear ring homomorphism

\begin{align}
\label{equ:morphwell}
\begin{split}
\Upsilon_i^{\underline{a}}:\mathrm{Sym}_{j\neq i}R[x_{1j},x_{2j},x_{3j}]^{(\rho-d_j)}&\to K[x_1,x_2,x_3]^{(d_i)}
\end{split}
\end{align}

induced by

\begin{equation}
\begin{pmatrix}
x_{1j}\\
x_{2j}\\
x_{3j}
\end{pmatrix}
\mapsto g_j^{-1}
\begin{pmatrix}
x_1\\
x_2\\
x_3
\end{pmatrix}.
\end{equation}
Note that \cref{equ:morphwell} is well-defined due to the fact that $\sum_{j=1}^{n+1} d_j=n\rho$, which yields $d_i=\sum_{j\neq i}(\rho-d_j)$.

 We denote by $\Sigma_i$ the subset of $\mathrm{Sym}_{j\neq i}R[x_{1j},x_{2j},x_{3j}]^{(\rho-d_j)}$ of polynomials $F$, such that the saturation $F'$ reduces to a polynomial $\overline{F}'$ modulo $\mathfrak{m}$ which satisfies
\begin{equation}
\overline{F}'\not\in \langle (x_{1j},x_{2j})_{j=1,\dots,n+1}\rangle.
\end{equation}

\begin{definition}
Let $f_1,\dots,f_{n+1}$ be polynomials with $f_i \in K[x_1,x_2,x_3]^{(d_i)}$ with degrees $d_i$ as above. Then we  say that the tuple $(f_1,\dots,f_{n+1})$ is \textit{$(\underline{d},\underline{a})-$admissible} if $f_i\in\Upsilon_i^{\underline{a}}(\Sigma_i)$.
\end{definition}

We are now ready to state our second main theorem.

\begin{theorem}
\label{thm:plane}
We fix natural numbers $n\ge2$ and $\rho$ and non-negative integers $d_1,\dots,d_{n+1}\leq \rho$, such that $\sum d_i=n\rho$. Furthermore, let $f_1,\dots,f_{n+1}$ be polynomials with $f_i \in K[x_1,x_2,x_3]^{(d_i)}$. \par 
Let $C\subset\bigcup_{i=1}^{n+1} D_+(f_i)\subset\mathbb{P}^2_{K}$ be a smooth plane curve, and let $\underline{a} \in \Delta(R) \cap R^{9(n+1)}$ be a choice of coefficients, such that $\mathfrak{C}^{\underline{a}}$ has star-like reduction and such that $(f_1,\dots,f_{n+1})$ is $(\underline{d},\underline{a})-$admissible.\par
Then there exists a vector bundle $\cE$  on $\mathfrak{C}^{\underline{a}}$ with generic fiber $E=\restr{\mathrm{Syz}(f_1,\dots,f_{n+1})(\rho)  }{C}$ whose special fiber is trivial on all reduced irreducible components of  $\mathfrak{C}^{\underline{a}}_k$. 
\end{theorem}

\begin{remark}
As explained in \cref{sec:pre}, this  implies that the bundle $E=\restr{\mathrm{Syz}(f_1,\dots,f_{n+1})(\rho)  }{C}$ is semistable of degree $0$ on $C$.
\end{remark}

\begin{proof}
By assumption, we find $F_i\in \Sigma_i$, such that $\Upsilon_i^{\underline{a}}(F_i)=f_i$. As 
\begin{equation}
\mathrm{Syz}(f_1,\dots,f_{n+1})\cong\mathrm{Syz}(\alpha_1f_1,\dots,\alpha_{n+1}f_{n+1})
\end{equation}
for $\alpha_i\in K^{\times}$ and $\Upsilon_i^{\underline{a}}(tF_i)=tf_i$, we may assume that $F_i$ is saturated with respect to $t$.

We observe $(g_1^{-1},\dots,g_{n+1}^{-1})^\ast F_i=f_i$ and recall that we chose $ C\subset\bigcup D_+(f_i)$. Furthermore, for each $j\in\{1,\dots,n+1\}$ we consider the following commutative diagram

\begin{equation}
\begin{tikzcd}
\mathbb{P}^2_K \arrow{dd}[swap]{g_j^{-1}}\arrow{rr}{(g_1^{-1},\dots,g_{n+1}^{-1})} \arrow[hookrightarrow]{dr} & & \prod_{i=1}^{n+1}\mathbb{P}(L) \arrow{dd}{p_j}\\
 &  \mathcal{M}_{\underline{a}}(\Gamma) \arrow{dr}{p_j} \arrow[hookrightarrow]{ur}{\textrm{emb.}}& \\
\mathbb{P}^2_K \arrow{rr}{\textrm{gen. fiber}}& & \mathbb{P}(L) 
\end{tikzcd}
\end{equation}

This implies that the generic fiber of the coherent sheaf $\cE'$ on $\mathcal{M}_{\underline{a}}(\Gamma)$ defined in \cref{equ:cohsheaf1}  by
\begin{equation}
0\to \mathcal{E}'\to \bigoplus_{j=1}^{n+1}p_j^\ast\mathcal{O}_{\mathbb{P}(L)}(\rho-d_j)\xrightarrow{(F_1,\dots,F_{n+1})} \bigotimes_{j=1}^{n+1}p_j^\ast\mathcal{O}_{\mathbb{P}(L)}(\rho-d_j),
\end{equation}
 is isomorphic to $\mathrm{Syz}(f_1,\dots,f_{n+1})(\rho)$ and in particular locally free on $\bigcup_{i=1}^{n+1} D_+(f_i)$.
  
 We define the coherent sheaf $\cE$ on $\mathfrak{C}^{\underline{a}}$ as the pullback of $\cE'$ via the embedding $\mathfrak{C}^{\underline{a}} \hookrightarrow \mathcal{M}_{\underline{a}}(\Gamma)$. Then $\cE$ has generic fiber $E= \restr{\mathrm{Syz}(f_1,\dots,f_{n+1})(\rho)  }{C}$, in particular, it is locally free on the generic fiber $C$ of  $\mathfrak{C}^{\underline{a}}$. Let us now consider points in the special fiber of $\left(\mathfrak{C}^{\underline{a}}\right)_k$ which is the union of $(n+1)$ irreducible components $C_i$. Recall that by star-like reduction we have $D_i = C_i^{\mathrm{red}} \simeq  \mathbb{P}^1_k$. As usual, we identify $D_i$ with its image in $\mathcal{M}_{\underline{a}}(\Gamma)$. Then  $D_i $ is given by the ideal 
\begin{equation}
\langle t,x_{1i},(x_{1l},x_{2l})_{l\neq i}\rangle.
\end{equation}
As $x_{1i},(x_{1l},x_{2l})_{l\neq i}$ vanish on $D_i$, we have $x_{3l}\neq 0$ for $l\neq i$. Therefore
\begin{equation}
z_{jl}=c_{j1}^{(l)}x_{1l}+c_{j2}^{(l)}x_{2l}+c_{j3}^{(l)}x_{3l}
\end{equation}
is equal to $\overline{c_{j3}}^{(l)}x_{3l}$ on $D_i$, which is non-zero for  $l\neq i$ and general choices of $c_{j3}^{(l)}$ in $R$. Thus, we obtain $D_i\subset\bigcap_{l \neq i}D(z_{jl})$.
Moreover, for any point $(x_{1i}:x_{2i}:x_{3i})\in D_i$, we have $(x_{1i}:x_{2i}:x_{3i})=(0:\eta_1^i:\eta_2^i)$ with $\eta_1^i\neq0$ or $\eta_2^i\neq0$. Therefore
\begin{equation}
z_{ji} = c_{j1}^{(i)}x_{1i}+c_{j2}^{(i)}x_{2i}+c_{j3}^{(i)}x_{3i}
\end{equation}
satisfies
$z_{ji}(0,\eta_1^i,\eta_2^i)=\overline{c_{j2}}^{(i)}\eta_1^i+\overline{c_{j3}}^{(i)}\eta_2^i\neq 0$ for general choices of $c_{j2}^{(i)},c_{j3}^{(i)}$ in $R$. \par

Hence every point in $D_i$ lies in an open neighbourhood of the form $U'_{\underline{c}}$ for a suitable (general) choice of $\underline{c}$. 

Recall that we have
\begin{equation}
F_i\in R\left[\left(x_{jl}\right)_{\substack{j=1,2,3;\\l\neq i}}\right]\quad\textrm{and}\quad \overline{F}_i \notin  \langle \left(x_{1l},x_{2l}\right)_{l\neq i}\rangle.
\end{equation}
As $x_{1l}=x_{2l}=0$ for $l\neq i$ on $D_i$, we have $\overline{F_i}_{|D_i}=\alpha \prod_{l\neq i}x_{3l}^{\rho - d_l}$ with $\alpha\neq 0$. As $x_{3l}\neq 0$ on $D_i$, we therefore obtain that $\overline{F_i}_{|D_i}$ is  a non-zero constant and thus
\begin{equation}
D_i\subset \{z\in\mathcal{M}_{\underline{a}}(\Gamma)\mid (F_i)_z\in\mathcal{O}^\times_{\mathcal{M}(\Gamma),z}\}\subset\{z\in\mathcal{M}_{\underline{a}}(\Gamma)\mid (H_i)_z\in\mathcal{O}^\times_{\mathcal{M}(\Gamma),z}\}=U_{\underline{c},i}''.
\end{equation}
This proves, that for every point in the special fiber $\mathfrak{C}^{\underline{a}}_k$ there exists a neighbourhood of the form $U_{\underline{c},i} =  U'_{\underline{c}} \cap U_{\underline{c},i}''$ in $\mathcal{M}_{\underline{a}}(\Gamma)$. Hence by \cref{lem:free} we find that $\mathfrak{C}^{\underline{a}}$ is contained in the locus of points in $\mathcal{M}_{\underline{a}}(\Gamma)$ where the morphism $\varphi$ is a surjective morphism of vector bundles with locally free kernel $\cE'$. Hence the pullback $\cE$ of $\cE'$ to $\mathfrak{C}^{\underline{a}}$ is a vector bundle with generic fiber $E$ sitting inside the short exact sequence 
\begin{equation}
0\to \cE \to \bigoplus_{j=1}^{n+1}\restr{p_j^\ast\mathcal{O}_{\mathbb{P}(L)}(\rho-d_j)}{\mathfrak{C}^{\underline{a}}}\xrightarrow{(F_1,\dots,F_{n+1})|_{\mathfrak{C}^{\underline{a}}}} \bigotimes_{j=1}^{n+1}\restr{p_j^\ast\mathcal{O}_{\mathbb{P}(L)}(\rho-d_j)}{\mathfrak{C}^{\underline{a}}} \to 0
\end{equation}
of vector bundles on $\mathfrak{C}^{\underline{a}}$.

Let us now study the restriction of $\cE$ to the reduced component $D_i$ in the special fiber. 
Let 
\begin{equation}
\restr{p_j}{D_i}:D_i\to\mathbb{P}_k^2
\end{equation}
be the projection map to the $j-$th component restricted to $D_i$ and observe that $\restr{p_j}{D_i}:D_i\xrightarrow{\sim}\mathbb{P}_k^1\subset\mathbb{P}^2_k$ by the assumption that $\mathfrak{C}^{\underline{a}}$ has star-like reduction. Moreover, we have $\restr{p_j}{D_i}(D_i)\cong pt$ for $i\neq j$. Thus, for $m\in\mathbb{Z}$ we have
\begin{equation}
\label{lem:bundles}
\restr{p_i}{D_i}^\ast(\mathcal{O}_{\mathbb{P}^2_k}(m))\cong \mathcal{O}_{\mathbb{P}^1_k}(m)\quad \textrm{and}\quad \restr{p_j}{D_i}^\ast(\mathcal{O}_{\mathbb{P}^2_k}(m))\cong \mathcal{O}_{\mathbb{P}^1_k}\quad\textrm{for}\quad j\neq i.
\end{equation}

Now $\restr{\cE}{D_i}$ is given as the kernel of the morphism
\begin{equation}
\label{equ:specfi}
\restr{\varphi_k}{D_i}:\bigoplus_{j=1}^{n+1}\restr{p_j^\ast\mathcal{O}_{\mathbb{P}(L)}(\rho-d_j)}{D_i}\xrightarrow{\restr{(\overline{F}_1,\dots,\overline{F}_{n+1})}{D_i}} \bigotimes_{j=1}^{n+1}\restr{p_j^\ast\mathcal{O}_{\mathbb{P}(L)}(\rho-d_j)}{D_i},
\end{equation}
which boils down to 

\begin{align}
\mathcal{O}_{\mathbb{P}^1_k}^{i-1}\oplus O_{\mathbb{P}^1_k}(\rho-d_i) \oplus \mathcal{O}_{\mathbb{P}^1_k}^{n-i}&\xrightarrow{} \mathcal{O}_{\mathbb{P}^1_k}(\rho-d_i)\\
(A_1,\dots,A_{n+1})&\mapsto \sum A_j\cdot \restr{\overline{F}_j}{D_i}.
\end{align}
Recall that $\restr{\overline{F}_j}{D_i}$ is a degree $\rho-d_i$ polynomial in $x_{2i},x_{3i}$ for $j\neq i$ and $\restr{\overline{F}_i}{D_i}$ is a non-zero constant polynomial. This implies that $\restr{\cE}{D_i}$ is isomorphic to $\mathcal{O}_{\mathbb{P}^1_k}^n$, which proves our theorem.

\end{proof}

This theorem immediately implies the following result.

\begin{corollary} Assume that $K$ is contained in $\overline{\mathbb{Q}}_p$ and $n\ge2$. Let  $f_i \in K[x_1,x_2,x_3]^{(d_i)}$ be $n+1$ polynomials of homogenous degrees $d_i \leq \rho$ satisfying $\sum d_i=n\rho$. Consider a connected smooth plane curve  $C$ contained in $\bigcup_{i=1}^{n+1} D_+(f_i)\subset\mathbb{P}^2_{K}$ and assume that there exists a choice of matrix coefficients $\underline{a} \in \Delta(R) \cap R^{9(n+1)}$, such that $\mathfrak{C}^{\underline{a}}$ has star-like reduction and such that $(f_1,\dots,f_{n+1})$ is $(\underline{d},\underline{a})-$admissible.\par
Then the base change of the syzygy bundle $E=\restr{\mathrm{Syz}(f_1,\dots,f_{n+1})(\rho)  }{C}$  to  $\mathbb{C}_p$ has strongly semistable reduction in the sense of \cref{def:strongred}.
\end{corollary}

\subsection{An example class}
In this subsection, we give a class of examples of $(\underline{d},\underline{a})-$admissible polynomials. In order to this, we fix $d_1,\dots,d_{n+1} \leq \rho$ with $\sum d_i=n\rho$ and homogeneous polynomials $h_1,\dots,h_{n+1}\in R[x_1,x_2,x_3]$ with $\mathrm{deg}(h_i)=d_i$.\par 
In particular, this yields that $\sum_{j\neq i} (\rho-d_j)=d_i$. Let $m\in R[x_1,x_2,x_3]$ be a monomial of degree $d_i$. We can then factor $m=m_1\cdots m_{i-1}m_{i+1}\cdots m_{n+1}$, where $m_j$ is a monomial of degree $\rho-d_j$. For each monomial $m$, we fix such a factorisation and map $m$ to the product $m_1\cdots m_{i-1} m_{i+1}\cdots m_{n+1}$.  By linear continuation, we obtain an injective morphism of $R$-modules

\[R[x_1,x_2,x_3]^{(d_i)}\hookrightarrow \mathrm{Sym}_{j\neq i}R[x_1,x_2,x_3]^{(\rho-d_j)}.\]
Composing with  the isomorphisms
\[
R[x_1,x_2,x_3]^{(\rho-d_j)}\xrightarrow{} R[x_{1j},x_{2j},x_{3j}]^{(\rho-d_j)}
\]
mapping $x_i$ to $x_{ij}$ yields the following injective map
\[R[x_1,x_2,x_3]^{(d_i)}\hookrightarrow\mathrm{Sym}_{j\neq i}R[x_{1j},x_{2j},x_{3j}]^{(\rho-d_j)}.\]

We denote by $\tilde{F}_i$ the image of $h_i$ in $\mathrm{Sym}_{j\neq i}R[x_{1j},x_{2j},x_{3j}]^{(\rho-d_j)}$ and define
\begin{equation}
\label{equ:pushfor}
F_i=\tilde{F}_i\left(g_1
\begin{pmatrix}
x_{11}\\
x_{21}\\
x_{31}
\end{pmatrix},\dots,
g_{n+1}
\begin{pmatrix}
x_{1\ n+1}\\
x_{2\ n+1}\\
x_{3\ n+1}
\end{pmatrix}
\right).
\end{equation}
We illustrate this in the following example.
\begin{example}
For $\rho=n+1$, we consider the tuple 
\begin{equation}
(h_1,\dots,h_{n+1})=(x_1^n,x_2^n,x_3^2,x_1x_2^{n-1}x_3,\dots,x_1^{n-2}x_2^2x_3).
\end{equation}
We observe that $(n+1-d_1,\dots,n+1-d_{n+1})=(1,1,n-1,0,\dots,0)$. Therefore, we may consider the factorisation
\begin{equation}
(h_1,\dots,h_{n+1})=(x_1(x_1^{n-1}),x_2(x_2^{n-1}),x_3x_3,x_1(x_2^{n-1})x_3,\dots,x_1(x_{1}^{n-3}x_{2}^2)x_3),
\end{equation}
which then yields the desired expression in $x_{ij}$
\begin{align}
\tilde{F}_1=x_{12}x_{13}^{n-1},\quad \tilde{F}_2=x_{21}x_{23}^{n-1},\quad \tilde{F}_3=x_{31}x_{32},\\
\tilde{F}_i=x_{12}\left(x_{13}^{i-4}x_{23}^{n+3-i}\right) x_{31}\quad\textrm{for}\quad i=4,\dots,n+1.
\end{align}
\end{example}
\vspace{15pt}

In the following proposition, we give a large class of examples which are $(\underline{d},\underline{a})-$admissible.

\begin{proposition}
\label{prop:ex}
Let $\underline{d}$ be an $(n+1)$- tuple of degrees with $\sum d_i=n\rho$ as above, and put
\begin{equation} f_i=\Upsilon_i^{\underline{a}}\left(\prod_{j\neq i} x_{3j}^{\rho-d_j}\right)\in K[x_1,x_2,x_3]
\end{equation}
for some choice of $\underline{a}$.
Further, let $h_1,\dots,h_{n+1}\in R[x_1,x_2,x_3]$ be any choice of homogeneous polynomials with $\mathrm{deg}(h_i)=d_i$. Then, the tuple $(f_1+h_1,\dots,f_{n+1}+h_{n+1})$ is $(\underline{d},\underline{a})-$admissible.
\end{proposition}

\begin{proof}
The proposition follows by observing that for $F_i$ as in \cref{equ:pushfor}, we have
\begin{equation}
\Upsilon_i^{\underline{a}}(\prod_{j\neq i} x_{3j}^{\rho-d_j}+F_i)=f_i+h_i
\end{equation}
and $\prod_{j\neq i} x_{3j}^{\rho-d_j}+F_i\in \Sigma_i$ by the nature of the matrices $g_l$. 
\end{proof}

\section{A concrete example on the Fermat curve}
\label{sec:brenner}

In this section, we prove that the bundle $\mathrm{Syz}(x_1^2,x_2^2,x_3^2)(3)$ on the Fermat curve $C=V(x_1^d+x_2^d+x_3^d)$  has potentially strongly semistable reduction in the sense of \cref{def:potstrongred}. This example is a concrete stable rank two and degree zero bundle for which the obvious reduction is not strongly semistable by \cite[Proposition 1 and Section 3]{br2}.  Hence it provides an obvious test case for the question if all semistable bundles of degree zero occur in a $p$-adic Simpson correspondence, as we have explained in the introduction. The strategy is as follows: We construct a finite covering $\alpha:C'\to C$, such that $\alpha^\ast\left(\mathrm{Syz}(x_1^2,x_2^2,x_3^2)(3)\right)$ has a model of the kind investigated in \cref{sec:plane}.

\vspace{\baselineskip}

We assume that the characteristic of $K$ is zero. Fix coefficients $\underline{a} \in \Delta(R) \cap R^{27}$ and write
\begin{equation}
M_l^{-1}=\begin{pmatrix}
a_{11}^{(l)} & a_{12}^{(l)} & a_{13}^{(l)}\\
a_{21}^{(l)} & a_{22}^{(l)} & a_{23}^{(l)}\\
a_{31}^{(l)} & a_{32}^{(l)} & a_{33}^{(l)}
\end{pmatrix}^{-1}
=\begin{pmatrix}
b_{11}^{(l)} & b_{12}^{(l)} & b_{13}^{(l)}\\
b_{21}^{(l)} & b_{22}^{(l)} & b_{23}^{(l)}\\
b_{31}^{(l)} & b_{32}^{(l)} & b_{33}^{(l)}
\end{pmatrix}
=B_l.
\end{equation}

We further observe that
\begin{align}
\tilde{P}_{ij;\underline{a}}\coloneqq&\Upsilon_l^{\underline{a}}(x_{3i}x_{3j})=t^{-4}(b_{31}^{(i)}x_{1}+b_{32}^{(i)}x_2+b_{33}^{(i)}x_3)(b_{31}^{(j)}x_{1}+b_{32}^{(j)}x_2+b_{33}^{(j)}x_3)\\
=&t^{-4}\left(b_{31}^{(i)}b_{31}^{(j)}x_{1}^2+b_{32}^{(i)}b_{32}^{(j)}x_2^2+b_{33}^{(i)}b_{33}^{(j)}x_3^2+(b_{31}^{(i)}b_{32}^{(j)}+b_{32}^{(i)}b_{31}^{(j)})x_1x_2+(b_{31}^{(i)}b_{33}^{(j)}\right.\\
&\left.+b_{33}^{(i)}b_{31}^{(j)})x_1x_3+(b_{32}^{(i)}b_{33}^{(j)}+b_{33}^{(i)}b_{32}^{(j)})x_2x_3\right)
\end{align}

and define

\begin{align}
\tilde{P}_{l;\underline{a}}=b_{11}^{(l)}x_1^2+b_{12}^{(l)}x_1x_2+b_{13}^{(l)}x_2^2+b_{21}^{(l)}x_2x_3+b_{22}^{(l)}x_3^2+b_{23}^{(l)}x_1x_3.
\end{align}

Finally, we define

\begin{equation}
P_{1;\underline{a}}=\tilde{P}_{23;\underline{a}}+t^4\tilde{P}_{1;\underline{a}},\quad P_{2;\underline{a}}=\tilde{P}_{13;\underline{a}}+t^4\tilde{P}_{2;\underline{a}},\quad
P_3=\tilde{P}_{12}+t^4\tilde{P}_3.
\end{equation}

Then, it is easy to see that for generic choices of $b_{ij}^{(l)}$ (which yield generic
choices of $a_{ij}^{(l)}$), we have $D_+(P_{1;\underline{a}})\cup D_+(P_{2;\underline{a}})\cup D_+(P_{3;\underline{a}})=\mathbb{P}^2_K$. Thus, we obtain a finite covering
\begin{align}
\alpha:\mathbb{P}^2_K&\to\mathbb{P}^2_K\\
\begin{pmatrix}
x_1\\
x_2\\
x_3
\end{pmatrix} 
&\mapsto
\begin{pmatrix}
P_{1;\underline{a}}(x_1,x_2,x_3)\\
P_{2;\underline{a}}(x_1,x_2,x_3)\\
P_{3;\underline{a}}(x_1,x_2,x_3)
\end{pmatrix}.
\end{align}

This restricts to a finite covering
\begin{equation}
\restr{\alpha}{C'_{\underline{a}}}: C'_{\underline{a}}\to C
\end{equation}
with $C'_{\underline{a}}=V(P_{1;\underline{a}}^d+P_{2;\underline{a}}^d+P_{3;\underline{a}}^d)$. The Jacobi criterion shows that smoothness of $V(P_{1;\underline{a}}^d+P_{2;\underline{a}}^d+P_{3;\underline{a}}^d)$ holds for generic choices of $\underline{b}$ and thus for generic choices of $\underline{a}$.  Therefore, as noted in \cref{rem:general}, we find that possibly after finite base change $C'_{\underline{a}}$ is a smooth irreducible curve  for general choices of $\underline{a}$. 

Moreover, we obtain
\begin{equation}
\alpha^\ast\left(\mathrm{Syz}(x_1^2,x_2^2,x_3^2)(3)\right)=\mathrm{Syz}(P_{1;\underline{a}}^2,P_{2;\underline{a}}^2,P_{3;\underline{a}}^2)(6).
\end{equation}

We observe that by \cref{prop:ex}, we have $P_{l;\underline{a}}^2\in\Upsilon_i^{\underline{a}}(\Sigma_i)$, as $P_{l;\underline{a}}^2=\tilde{P}_{ij;\underline{a}}^2+2t^4\tilde{P}_{ij;\underline{a}}\tilde{P}_{l;\underline{a}}+t^8\tilde{P}_{l;\underline{a}}^2$ and
\begin{equation}
\tilde{P}_{ij;\underline{a}}^2=\Upsilon_i^{\underline{a}}(x_{3i}^2x_{3j}^2)\quad \textrm{and} \quad 2t^4\tilde{P}_{ij;\underline{a}}\tilde{P}_{l;\underline{a}}+t^8\tilde{P}_{l;\underline{a}}^2\in R[x_1,x_2,x_3].
\end{equation}

We now want to apply \cref{thm:plane} to $\mathrm{Syz}(P_{1;\underline{a}}^2,P_{2;\underline{a}}^2,P_{3;\underline{a}}^2)(6)$ on $C'_{\underline{a}}$. What is left to show is that there exists a choice of coefficients $\underline{a}$, such that Mustafin model $\mathfrak{C}'^{\underline{a}}$ of $C'_{\underline{a}}$ has star-like reduction. The key problem here is that $C'_{\underline{a}}$ depends on the choice of $\underline{a}$. We resolve this problem in the proof of the following theorem.

\begin{theorem}
\label{thm:bren} Let $K$ be a discretely valued field of characteristic zero and consider the Fermat curve $C=V(x_1^d+x_2^d+x_3^d) \subset \mathbb{P}_K^2$  and the  vector bundle $E = \mathrm{Syz}(x_1^2,x_2^2,x_3^2)(3)|_C$. Then there exists a finite extension $L$ of $K$ and a degree two cover $\alpha: C' \rightarrow C_L$ by a connected, smooth, projective curve $C'$ together with a model $\mathfrak{C}'$ of $C'$ over $R_L$ and a vector bundle $\cE$ on $\mathfrak{C}'$ with generic fiber $\alpha^\ast E_L$ such that the special fiber of $\cE$ is trivial on all reduced irreducible components of $\mathfrak{C}_k'$.

\end{theorem}
This immediately  implies the following result. 
\begin{corollary}
The vector bundle $\mathrm{Syz}(x_1^2,x_2^2,x_3^2)(3)$ on the Fermat curve over $\mathbb{C}_p$ has potentially strongly semistable reduction in the sense of \cref{def:potstrongred}.
\end{corollary}
\begin{proof}
We investigate the pullback of $\mathrm{Syz}(P_{1;\underline{a}}^2,P_{2;\underline{a}}^2,P_{3;\underline{a}}^2)(6)$ to $C'_{\underline{a}}$. \par 
We have already seen that $(P_{1;\underline{a}}^2,P_{2;\underline{a}}^2,P_{3;\underline{a}}^2)$ is $((4,4,4),\underline{a})-$admissible. Thus, the fact that $\mathrm{Syz}(P_{1;\underline{a}}^2,P_{2;\underline{a}}^2,P_{3;\underline{a}}^2)(6)$ on $C'_{\underline{a}}$ has (after a finite base extension) a model as in our claim follows from \cref{thm:plane}, if we can choose $\underline{a}$ so that  additionally $\mathfrak{C}'^{\underline{a}}$ has star-like reduction. More precisely, considering the embedding of $C'_{\underline{a}}$ into $\mathcal{M}_{\underline{a}}(\Gamma)$, we obtain a model $\mathfrak{C'}^{\underline{a}}$ of $C'_{\underline{a}}$. We show that for general $\underline{a}$ the special fiber $\left(\mathfrak{C'}^{\underline{a}}\right)_k$ decomposes as in \cref{equ:comp}. In order to see this, we proceed similarly as in the proof of \cref{thm:spec}.
\begin{enumerate}[leftmargin=*]
\item We first prove an analog of \cref{lem:step1}, i.e. assuming that for general $\underline{a}$, the irreducible components of $\mathfrak{C'}^{\underline{a}}_k$ are only contained in primary components of $\mathcal{M}_{\underline{a}}(\Gamma)_k$, then the model $\mathfrak{C'}^{\underline{a}}$ has star-like reduction for general $\underline{a}$.\par 
In order to see this, we first study the model obtained by embedding the curve into a single projective space $\mathbb{P}(L)$. Let $\mathfrak{C'}^{(l)}_{\underline{a}}$ be the model of $C'_{\underline{a}}$ obtained by embedding $C'_{\underline{a}}$ into $\mathbb{P}(L)$ via $g_l$. Consider the polynomial 
\begin{align}
t^{4d}S^{(l)}_{\underline{a}}&=t^{4d}\left(P_{1;\underline{a}}^d+P_{2;\underline{a}}^d+P_{3;\underline{a}}^d\right)\\
&\left(a_{11}^{(l)}x_1+ta_{12}^{(l)}x_{2}+t^2a_{13}^{(l)}x_{3},a_{21}^{(l)}x_1+ta_{22}^{(l)}x_{2}+t^2a_{23}^{(l)}x_{3},a_{31}^{(l)}x_1+ta_{32}^{(l)}x_{2}+t^2a_{33}^{(l)}x_{3}\right).
\end{align}
We  observe that $t^{4d}S^{(l)}_{\underline{a}}\in R[x_1,x_2,x_3]$. To fix ideas, let us suppose $l=1$. By Cramer's rule, we have
\begin{align}
a_{11}^{(1)}=\frac{1}{\mathrm{det}(B_1)}\left(b_{22}^{(1)}b_{33}^{(1)}-b_{32}^{(1)}b_{23}^{(1)}\right)\\
a_{21}^{(1)}=\frac{1}{\mathrm{det}(B_1)}\left(b_{31}^{(1)}b_{23}^{(1)}-b_{21}^{(1)}b_{33}^{(1)}\right)\\
a_{31}^{(1)}=\frac{1}{\mathrm{det}(B_1)}\left(b_{21}^{(1)}b_{32}^{(1)}-b_{31}^{(1)}b_{22}^{(1)}\right)
\end{align}
After substituting these expressions into $t^{4d}S^{(1)}_{\underline{a}}$, we scan $\overline{\mathrm{det}(B_1)}\cdot\overline{t^4S^{(1)}_{\underline{a}}}$ for the monomial
\begin{equation}
\label{equ:monomial}
\left(\overline{b}_{31}^{(2)}\right)^d\left(\overline{b}_{31}^{(3)}\right)^d\left(\overline{b}_{22}^{(1)}\right)^{2d}\left(\overline{b}_{33}^{(1)}\right)^{2d}x_1^{2d}.
\end{equation}
Since \cref{equ:monomial} is a monomial of degree $d$ in $b_{ij}^{(2)}$ and $b_{ij}^{(3)}$ but of degree $2d$ in $b_{ij}^{(1)}$, it can only occur in the reduction of 
\begin{align}
&\mathrm{det}(B_1)\cdot (t^4\tilde{P}_{23})^d\\
&\left(a_{11}^{(l)}x_1+ta_{12}^{(l)}x_{2}+t^2a_{13}^{(l)}x_{3},a_{21}^{(l)}x_1+ta_{22}^{(l)}x_{2}+t^2a_{23}^{(l)}x_{3},a_{31}^{(l)}x_1+ta_{32}^{(l)}x_{2}+t^2a_{33}^{(l)}x_{3}\right).
\end{align}
A patient term-by-term analysis shows that  \cref{equ:monomial}  only occurs once in this expression and therefore cannot cancel in $\overline{\mathrm{det}(B_1)}\cdot\overline{t^{4d}S^{(1)}_{\underline{a}}}$.
Hence $\overline{\mathrm{det}(B_1)}\cdot\overline{t^{4d}S^{(1)}_{\underline{a}}} = \gamma(\underline{\overline{b}} )x_1^{2d}$ for a non-zero rational function $\gamma$ in the coefficients $\underline{\overline{b}}$. 
 For general choices of $b_{ij}^{(l)}$, we therefore have $\gamma(\underline{\overline{b}})\neq 0$. We proceed in a similar way for $l=2$ and $l=3$. Thus, analogously to the proof of \cref{lem:step1}, we  see that for general choices of $\underline{a}$ the model $\mathfrak{C}'^{(l)}_{\underline{a}}$ is cut out by
$t^{4d}S^{(l)}_{\underline{a}}$,  and that the special fiber is 
\begin{equation}
\left(\mathfrak{C'}^{(l)}_{\underline{a}}\right)_k=\mathrm{Spec}\left(\faktor{k[x_1,x_2,x_3]}{\langle x_1^{2d}\rangle}\right).
\end{equation}
Therefore, we obtain completely analogously to the proof of \cref{lem:step1}, that under the assumption that for general $\underline{a}$ all components of $\left(\mathfrak{C'}^{\underline{a}}\right)_k$ are contained in primary components of $\mathcal{M}_{\underline{a}}(\Gamma)_k$, the models $\mathfrak{C'}^{\underline{a}}$ have star-like reduction for general $\underline{a}$.
\item Now, we show that for general $\underline{a}$ all irreducible components of $\left(\mathfrak{C'}_{\underline{a}}\right)_k$ lie in primary components of $\mathcal{M}_{\underline{a}}(\Gamma)_k$. In order to prove this, we work with the same algebraic set-up as in the proof of \cref{lem:step2}.\par 
Recall the notation 
\begin{equation}
\mathfrak{g}_i=\begin{pmatrix}
A_{11}^{(i)} & A_{12}^{(i)} & A_{13}^{(i)}\\
A_{21}^{(i)} & A_{22}^{(i)} & A_{23}^{(i)}\\
A_{31}^{(i)} & A_{32}^{(i)} & A_{33}^{(i)}
\end{pmatrix}\cdot
\begin{pmatrix}
1 & 0 & 0\\
0 & t & 0\\
0 & 0 & t^2
\end{pmatrix}.
\end{equation}

We further denote
\begin{equation}
\begin{pmatrix}
A_{11}^{(i)} & A_{12}^{(i)} & A_{13}^{(i)}\\
A_{21}^{(i)} & A_{22}^{(i)} & A_{23}^{(i)}\\
A_{31}^{(i)} & A_{32}^{(i)} & A_{33}^{(i)}
\end{pmatrix}^{-1}
=
\begin{pmatrix}
B_{11}^{(i)} & B_{12}^{(i)} & B_{13}^{(i)}\\
B_{21}^{(i)} & B_{22}^{(i)} & B_{23}^{(i)}\\
B_{31}^{(i)} & B_{32}^{(i)} & B_{33}^{(i)}
\end{pmatrix}
=\mathcal{B}_i.
\end{equation}

We now consider
\begin{align}
\tilde{\mathcal{P}}_{ij}\coloneqq & t^{-4}\left(B_{31}^{(i)}B_{31}^{(j)}y_{1}^2+B_{32}^{(i)}B_{32}^{(j)}y_2^2+B_{33}^{(i)}B_{33}^{(j)}y_3^2+(B_{31}^{(i)}B_{32}^{(j)}+B_{32}^{(i)}B_{31}^{(j)}y_1y_2\right.\\
&+(B_{31}^{(i)}B_{33}^{(j)}+B_{33}^{(i)}B_{31}^{(j)})y_1y_3+\left.(B_{32}^{(i)}B_{33}^{(j)}+B_{33}^{(i)}B_{32}^{(j)})y_2y_3\right),\\
\tilde{\mathcal{P}}_{l}= & B_{11}^{(l)}y_1^2+B_{12}^{(l)}y_1y_2+B_{13}^{(l)}y_2^2+B_{21}^{(l)}y_2y_3+B_{22}^{(l)}y_3^2+B_{23}^{(l)}y_1y_3.
\end{align}
and
\begin{equation}
\mathfrak{f}=(\tilde{\mathcal{P}}_{23}+t^4\tilde{\mathcal{P}}_{1})^d+(\tilde{\mathcal{P}}_{13}+t^4\tilde{\mathcal{P}}_{2})^d+(\tilde{\mathcal{P}}_{12}+t^4\tilde{\mathcal{P}}_{3})^d.
\end{equation}
The last equation defines a curve $\mathcal{C}=V(\mathfrak{f})\subset\mathbb{P}(\mathcal{V})$. We take the closure of the image of $\mathcal{C}$ in $\mathcal{N}(\Gamma')$ via \cref{equ:mustemb} and endow it with the reduced scheme structure to obtain a scheme $\mathscr{C}$. As in the previous section, we denote for all  $\underline{a} \in \Delta(R)$ the homomorphism $\lambda_{\underline{a}}: \mathcal{R} \rightarrow R$ mapping $A_{ij}^{(l)}$ to $a_{ij}^{(l)}$ and by  $\Lambda_{\underline{a}}$ the homomorphism obtained by composing $\lambda_{\underline{a}}$ with $R\hookrightarrow K$. The base change of $\mathscr{C}$ along $\Lambda_{\underline{a}}$ is by construction isomorphic to $C'_{\underline{a}}$ for generic $\underline{a}$. Hence $\Lambda_{\underline{a}}^{\ast}\mathcal{C}$ is generically an irreducible and reduced curve, and after possibly passing to a finite field extension, the same condition holds for general $\underline{a}$.  Note that the curve $\mathcal{C}$  is irreducible by \cref{lemma-fibers}.\par
Now, the key observation is that $\tau_{\sigma}(I(\mathcal{C}))=I(\mathcal{C})$. Then, after passing to a finite field extension, we use  \cref{lem:step2} to find an open set $W\subset\mathrm{Spec}(\mathcal{R}')$, such that for all $y \in W(k)$ the irreducible components of $\mathscr{C}'_y$ are contained in primary components of the special fiber of the Mustafin variety $\mathcal{N}'(\Gamma')_y=\mathcal{M}_{\underline{a}}(\Gamma)_k$ for coefficients $\underline{a} $ in $\Delta(R) \cap R^{27} $  lifting $y$. 
\end{enumerate}
Now, it follows immediately from (1) and (2) that for general $\underline{a}$, the model $\mathfrak{C'}^{\underline{a}}$ has star-like reduction. Therefore, the theorem follows with the help of \cref{thm:plane} applied to the bundle $\mathrm{Syz}(P_{1;\underline{a}}^2,P_{2;\underline{a}}^2,P_{3;\underline{a}}^2)(6)$.
\end{proof}

\bibliographystyle{acm}


\end{document}